\documentclass[a4paper]{amsart}
\usepackage{enumerate}
\usepackage{mathtools}
\usepackage{amssymb}
\usepackage{xcolor}
\usepackage{pgfplots}
\usepackage{ifthen}
\newboolean{pedro}
\setboolean{pedro}{true}
\newcommand{\pedro}{\ifthenelse{\boolean{pedro}}{\color{magenta}
    \setboolean{pedro}{false}}{\color{black}\setboolean{pedro}{true}}}    

\newboolean{jose}
\setboolean{jose}{true}
\newcommand{\jose}{\ifthenelse{\boolean{jose}}{\color{blue}
    \setboolean{jose}{false}}{\color{black}\setboolean{jose}{true}}}

\newboolean{joseN}
\setboolean{joseN}{true}
\newcommand{\joseGreen}{\ifthenelse{\boolean{joseN}}{\color{green}
    \setboolean{joseN}{false}}{\color{black}\setboolean{joseN}{true}}}  

\newboolean{joseQ}
\setboolean{joseQ}{true}
\newcommand{\joseQ}{\ifthenelse{\boolean{joseQ}}{\color{red}
    \setboolean{joseQ}{false}}{\color{black}\setboolean{joseQ}{true}}}

\title[Rational rank of solutions]
{The Rational rank of the support of
  generalized power series solutions of differential and
  $q$-difference equations}
\author{J. Cano and P. Fortuny Ayuso}
\address{J. Cano: Facultad de Ciencias. Universidad de Valladolid.
  47011-Valladolid Spain}
\email{jcano@uva.es}

\address{P. Fortuny Ayuso: Universidad de Oviedo. 33007-Oviedo, Spain}
\email{fortunypedro@uniovi.es}

\subjclass{34M25 (Primary), 30D05 (Secondary)}
\keywords{Rational rank; Generalized Power Series Solution;
  Ordinary Differential Equation; $q$-difference Equation;  Newton
  Polygon Method; Convergence criteria.}

\theoremstyle{definition}
\newtheorem{definition}{Definition}

\newtheorem{theorem}{Theorem}
\newtheorem{proposition}{Proposition}
\newtheorem{lemma}{Lemma}
\newtheorem{corollary}{Corollary}
\newtheorem{remark}{Remark}
\newtheorem{procedure}{Procedure}

\DeclareMathOperator{\NIC}{NIC}
\DeclareMathOperator{\supp}{supp}
\DeclareMathOperator{\height}{ht}
\DeclareMathOperator{\ord}{ord}
\DeclareMathOperator{\Top}{Top}
\DeclareMathOperator{\Bot}{Bot}

\newcommand{\llangle}{\left\langle}
\newcommand{\rrangle}{\right\rangle}
\newcommand{\angles}[1]{\left\langle {#1}\right\rangle}

\def\mina{\lambda}
\newcommand{\dmu}[1]{\delta_{#1}}
\newcommand{\dmuk}[2]{\delta_{#1}^{(#2)}}
\newcommand{\shift}{\varepsilon}
\newcommand{\DD}{\mathcal{D}}

\begin{document}
\begin{abstract}
  Given a differential or $q$-difference equation $P$ of order $n$, we
  prove that the set of exponents of a generalized power series
  solution has its rational rank bounded by the rational rank of the
  support of $P$ plus $n$. We also prove that when the support of the
  solution has maximum rational rank, it is convergent. Using the
  Newton polygon technique, we show also that in the maximum rational
  rank case, an initial segment can always be completed to a true
  solution. The techniques are the same for the differential and the
  $q$-difference case.
\end{abstract}
\maketitle

\section{Introduction}

An important class of solutions of differential and $q$-difference
equations are formal power series.
Their relevance comes from algebraic to geometric, analytical, combinatorial and
logical point of views.  Denef and Lipschitz's paper \cite{Denef1984} deals with the
logical questions; Singer's and Grigoriev's works study the possibility of finding these
solutions, and some of their convergence properties \cite{SINGER199059,Grigoriev-Singer};
the excellent book on combinatorics by Flajolet and Sedgewick \cite{flajolet2009analytic}
contains a complete section devoted to the relation between solutions of polynomial
differential equations and combinatorics; and one cannot forget the relation between power
series, derivations, Hardy fields, o-minimal structures and related topics
\cite{Matusinski2006,CanoMoussuRolin+2005+107+141,10.1112/plms/pdm016}.
Finally, any survey of the
bibliography would be incomplete without citing van der Hoeven's works on transseries
\cite{van2006transseries} and the \emph{magnum opus}
\cite{aschenbrenner2017asymptotic}.

Grigoriev and Singer in~\cite{Grigoriev-Singer}
introduce the field $\Omega$ of generalized power
series with complex coefficients and real exponents:
series $\sum_{i=0}^{\infty}a_i\,x^{\nu_i}$, where
$a_i\in \mathbb{C}$, $\nu_i\in \mathbb{R}$, $v_0<\nu_1<\cdots$ and
$\lim \nu_i=\infty$,
and prove that if
$y(x)=\sum_{i=0}^{\infty}a_i\,x^{\nu_i}\in \Omega$ is a solution of a non trivial
polynomial ordinary differential equation $P(y,y',\ldots,y^{(n)})=0$ with coefficients in
the field of formal Laurent series $\mathbb{C}((x))$,
then its support
$\supp y(x)=\{\nu_i\mid a_i\neq 0\}$ (the set of exponents with non-zero coefficient),
generates a finite $\mathbb{Z}$-module but they provide no information
on its rational rank. Our first aim (Theorem~\ref{the:main1})
is to prove in an elementary way that,
modulo the support of the equation, 
this rational rank is at most the
order of the equation.
Our arguments work almost word by word
for $q$-difference equations.
Our proof of Theorem~\ref{the:main1} allows us
to show that in both cases, when the rational rank of the
solution reaches its maximum possible value, then the solution is
necessarily convergent (assuming $P$ is a polynomial), which is
Theorem~\ref{the:main2}.
We make use of the convergence  results of Gontsov, Goryuchkina
and Lastra~\cite{gontsov-goryuchkina2015,Gontsov2022579}.
Theorem~\ref{the:main3} addresses the possibility of completing an initial
segment $s_0(x)=\sum_{1\leq i\leq k} c_i\,x^{\mu_i}$ to an actual power series 
$s(x)$ solution of the differential or $q$-difference equation $P=0$.
Finally, Theorem~\ref{th:AutFirstOrderODEs} is the consequence of applying the
previous results to the case of autonoumous first order differential
equations, providing new proofs about the existence and convergence
of Puiseux solutions of these equations given in \cite{Cano2022137}.

Other authors have proved related finiteness results in other contexts
\cite{Matusinski2006,CanoMoussuRolin+2005+107+141,10.1112/plms/pdm016,
  Cano-Fortuny-Asterisque,Cano-Fortuny-q-diff-2022},
but among these, none gives effective bounds for the rational rank of
the support of the solutions.
The unique effective bound we have knowledge of is by
van der Hoeven, in \cite{van2006transseries} (Corollary
8.38), who shows in the language of transseries that
if $P(y,y^{\prime},\ldots,y^{(n)})=0$ is
an algebraic differential equation and $f$ is a transseries
solution, then \emph{essentially} the rational rank of the semigroup
of monomials of $f$ can only
increase by the order~$n$ of $P$. This is the same result as our
Theorem~\ref{the:main1} in the case of differential equations,
and over a field which neither contains nor is contained in $\Omega$. Our
arguments have the advantage of simplicity and conciseness.

Rosenlicht, in the context of Hardy
fields~\cite{Rosenlicht1983659} shows that the rank
(the number of comparability classes) of the differential field
generated by a solution of the differential equation is bounded by the
order of the equation. We, however, work over $\Omega$ which has
rank~$1$, and his results do not give insight about the rational rank of the support
of the solution.

In this paper we work indistinctly with ordinary differential
equations and $q$-difference equations.
Given a polynomial $P\in \Omega[y_0,y_1,\ldots,y_n]$, the equation
$P=0$ means $P(y(x),y'(x),\ldots,y^{(n)}(x))=0$, with the 
assumption that $y'(x)$ refers to either
the usual derivation operator $y'(x)=\frac{d\, y(x)}{d\, x}$,
the Euler derivation
$y'(x)=x\,\frac{d\, y(x)}{d\, x}$ or
$y'(x)=y(q\, x)$, for some $q\in\mathbb{C}$, with $|q|\neq 1$. We do
not consider mixed equations.

Given $P\in \Omega[y_0,y_1,\ldots,y_n]$ 
we denote by $\supp P$  the union of the supports of the coefficients
of $P$, and if $E$ is a subset of $\mathbb{R}$ then
$\left\langle E \right\rangle$ is the
$\mathbb{Q}$-linear subspace of $\mathbb{R}$
generated by $E$.
Our first result is Theorem~\ref{the:main1} which states that if
$s(x)\in \Omega$ is a solution of $P=0$ then
  \begin{equation}\label{eq:IntroDesigualdadTh1}
    \dim_{\mathbb{Q}} \frac{\left\langle \supp s(x) \cup  \supp P
      \right\rangle}
    {\llangle \supp P \rrangle}
    \leq
    n .
  \end{equation}
In particular, if the coefficients of $P$ are formal power series with integral
exponents, then the maximal number of rationally independent irrational exponents that
can appear in $\supp s(x)$ is the order $n$ of the equation $P=0$.


The second result of this paper address the question of the convergence of the solution
$s(x)$ provided $P$ is a polynomial in $x,y_0,\ldots,y_n$. It is well known that the
convergence of a solution is not guaranteed: Euler's example $y-x^{2}\,y'-x=0$ has as
solution the formal power series $\varepsilon(x)=\sum_{n\geq 0} n!\, x^{n+1}$. B. Malgrange
in~\cite{Malgrange:1989} gives a sufficient condition for a formal power series solution
$s(x)$ of a polynomial differential equations $P=0$ to be convergent in terms of the
linearized differential operator along the solution,
$L=\sum_{i=0}^{n}\frac{\partial P}{\partial y_i}(s(x),\ldots,s^{(n)}(x))\,\delta^{i}$: if
$L$ has a regular singularity at the origin then the solution $s(x)$ is convergent.  This
criterion has been extended in the differential case by R.~Gontsov and
I.~Goryunchkina
in~\cite{gontsov-goryuchkina2015}
and in the $q$-difference case by R.~Gontsov and I.~Goryunchkina and A. Lastra
in~\cite{Gontsov2022579} 
for power series solutions
with complex exponents.
In Theorem~\ref{the:main2} of this paper we provide a new
sufficient condition for the convergence of the solution $s(x)$ which only depends on the
order $n$ of the differential or $q$-difference equation $P=0$, and on the support of the
solution. Specifically, if $\langle\supp s(x)\rangle$
has the maximal possible rational rank
(Equation~\eqref{eq:IntroDesigualdadTh1} is an equality)
then $s(x)$ is necessarily convergent.
In particular, if $P$ is of order one and with constant coefficients,
any solution of the equation $P=0$ is convergent.
Theorem~\ref{the:main2} provides a negative criterion for a series
to be a solution of a differential or $q$-difference equation. For
instance,  neither the series
$x^{\pi}+\varepsilon(x)$ nor $x^{\pi}\,\varepsilon(x)$ (where
$\varepsilon(x)$ is as above) is a solution
of a nonlinear first order differential (or $q$-difference) equation
with convergent coefficients.

Our third main result addresses the question of the possibility of completing an initial
segment $s_0(x)=\sum_{1\leq i\leq k} c_i\,x^{\mu_i}$ to an actual power series solution
$s(x)$ solution of the equation $P=0$. Given an equation $P=0$,
the Newton polygon method
(\cite{Grigoriev-Singer,Cano-Fortuny-Asterisque,Cano-Fortuny-q-diff-2022}),
provides, for any positive integer $k$,
a finite family of necessary initial conditions 
$\NIC_k(P)$, such that the first $k$ coefficients and exponents
of any solution $\sum_{i=1}^{\infty}a_i\,x^{\nu_i}$ of $P=0$ satisfy
$\NIC_k(P)$.  We say that $s_0(x)$ is an \emph{admissible initial
  segment} for $P=0$ if
its coefficients and exponents satisfy $\NIC_k(P)$.

In the algebraic case (i.e. $n=0$), Puiseux's Theorem shows that any admissible initial
segment for $P=0$ is the truncation of an actual solution.  This is no longer true for
differential and $q$-difference equations:
even in order and degree one, there are examples
\cite{CanoJ2,barbe2020qalgebraic}
of admissible initial segments which cannot be completed to a solution.
Section 3 of~\cite{Grigoriev-Singer} introduces the concept of \emph{stabilization} as a
criterion which guarantees the existence (and uniqueness) of an actual solution with
initial segment $s_0(x)$. In Theorem~\ref{the:main3}
we introduce a criterion of a different
nature: if the support of an admissible polynomial $s_0(x)$ for $P=0$
has $n$ linearly independent irrational numbers over $\mathbb{Q}$
(modulo  $\langle\supp P\rangle$) then
$s_0(x)$ is in fact the truncation of an actual solution of
$P=0$. Indeed, there can be more than one solution $s(x)$ with the
same truncation $s_0(x)$, unlike the stabilization criterion, for
which the solution is unique
(see example in Section~\ref{sec:example}).
In particular, if $P=P(x,y)=0$ is a
polynomial in $x,y$ (i.e. the equation of an algebraic curve), then the hyphothesis of
Thereom~\ref{the:main3} always hold, so that it may be understood as a
kind of
generalization of Puiseux's Theorem to differential and $q$-difference
equations.

Section~\ref{sec:NotationStatements} introduces the context and a technical result from
which the two main theorems follow in Section \ref{sec:twotheorems}. Then, in
Section~\ref{sec:NewtonPolygon} we recall the essential notions and results on the Newton polygon
process (see \cite{Grigoriev-Singer,CanoJ,CanoJ2,Cano-Fortuny-Asterisque,Cano-Fortuny-q-diff-2022,barbe2020qalgebraic}), which we use
to prove the technical Proposition~\ref{prop:main}, from which Theorem \ref{the:main3}
follows straightforwardly.
In subsection~\ref{subsec:autFirstOrderDiffODEs} we apply the
previous results to the case of autonomous ordinary differential equations.
In Section~5 we give a detailed example to
illustrate our main results.

Notice that instead of generalized power series in $x$, we could have worked with
generalized power series in $x^{-1}$. The results and proves are the
same with the obvious modifications.

\section{Notation and a Technical Lemma}\label{sec:NotationStatements}
Following \cite{Grigoriev-Singer}, we denote by
$\Omega$ the field of formal power series $s(x)=\sum c_ix^{\mu_i}$
with complex coefficients and real exponents, such
that
$\lim_{i\rightarrow\infty}\mu_{i}=\infty$, where we always assume
that $\mu_1<\mu_2<\cdots$.
The \emph{support} of a power series $s(x)$ is the set of exponents with non-zero coefficient:
$\supp s(x) = \{\mu_i\ :\ c_i\neq 0\}$.
Given an additive subgroup
$\Gamma$ of $\mathbb{R}$
we denote
by $\mathbb{C}((x^{\Gamma}))\subset \Omega$ the subfield  of
series $s(x)\in \Omega$ whose support is contained in $\Gamma$.
We also denote by
$\hat{K}=\mathbb{C}((x^{\mathbb{Z}}))$ the
field of fractions of the ring of formal power series.

We are going to deal
with differential and $q$-difference equations simultaneously,
hence it should be more convenient to use the Euler derivative in the
case of differential equations, but some of our results are specific
for the case of differential equations with the ordinary derivation, in
particular, the case of autonomous equations. Hence we deal
simultaneously with differential equations in terms of the ordinary
differential operator, in terms of the Euler differential operator and with
$q$-difference equations, with $|q|\neq 1$.

Thorought this paper $P$ will denote a polynomial $P(y_0,\ldots, y_{n})$ in the
indeterminates $y_0,y_1,\ldots,y_n$ with coefficients in $\mathbb{C}((x^{\Gamma}))$ for
$\Gamma\subset \mathbb{R}$ as above.  Given $s(x)\in \Omega$, we set
$P(s(x))=P(s(x),s'(x),\ldots,s^{(n)}(x))$, where $s'(x)$ is either
the ordinary derivative $\frac{d\,s(x)}{d\,x}$, the
Euler derivative $x\,\frac{d\,s(x)}{d\,x}$ or the $q$-difference
operator $s(q\,x)$, and $s^{(\kappa)}(x)$ is the $\kappa$-th iteration.
We do not consider mixed equations.

If $s(x)=\sum_{i=0}^{\infty} c_i\, x^{\mu_i}$, and $\kappa=0,\ldots,n$
then we write the $\kappa$-th iteration of the operator $'$ as:
\begin{displaymath}
  s^{(\kappa)}(x)=\sum_{i=0}^{\infty} c_i\, \dmuk{\mu_i}{\kappa}
  \,x^{\mu_i-\shift \kappa},
\end{displaymath}
where $\shift=1$ in the case of differential equations with
the ordinary differential operator $\frac{d\,}{d\,x}$ or $\shift=0$ 
for differential equations with the Euler differential operator, and
$q$-difference equations;
the symbol $\dmuk{\mu}{\kappa}$ is defined as
$\dmuk{\mu}{\kappa}=\mu\,(\mu-1)\cdots(\mu-\kappa+1)$ for
differential equations with the operator $\frac{d\,}{d\,x}$,
$\dmuk{\mu}{\kappa}=\mu^{\kappa}$ for
differential equations with the Euler differential operator
$x\,\frac{d\,}{d\,x}$,
and $\dmuk{\mu}{\kappa}=q^{\kappa\,\mu}$ for $q$-difference
equations.
In particular $\dmuk{\mu}{0}=1$ in all cases, and we denote
$\dmu{\mu}=\dmuk{\mu}{1}$.

The expression $P=0$ denotes the corresponding
equation of order $n$.  A \emph{solution} of $P=0$ is a power series $s(x)$ such
that $P(s(x),s^{\prime}(x),\ldots,s^{(n)}(x))=0$.
The \emph{support} of $P$ is the union of
the supports of its coefficients.

Our first two results rely on the following elementary fact about derivations with respect
to \emph{transcendent monomials}.
Fix $\Gamma$ an additive subgroup of $\mathbb{R}$
and consider a family of exponents
$\mu_1,\ldots,\mu_m\in \mathbb{R}$ whose classes modulo $\llangle\Gamma\rrangle$,
$\bar{\mu}_1,\ldots,\bar{\mu}_m\in \mathbb{R}/\langle \Gamma\rangle $ are
$\mathbb{Q}$-linearly independent. Let
$\Gamma'=\langle \Gamma\cup\{\mu_1,\ldots,\mu_m\}\rangle$. By definition, any
$\alpha\in \Gamma'$ can be written uniquely as
  \[
    \alpha=[\alpha]_\Gamma+\sum_{j=1}^{m} [\alpha]_{\mu_j}\,\mu_j,
  \]
  where $[\alpha]_\Gamma\in \llangle\Gamma\rrangle$ and $[\alpha]_{\mu_j}\in \mathbb{Q}$.
  \begin{lemma}\label{le:derivations-x^mu}
    With the above notation and conditions, for any $j=1,\ldots,m$,
    there exists
    derivation  $\DD_j$ on
    $\mathbb{C}((x^{\Gamma'}))$
    such
    that
  \begin{displaymath}
   \DD_j\left(
     \sum_{i=1}^{\infty}c_i\,x^{\alpha_i}\right)=
     \sum_{i=1}^{\infty}c_i\,[\alpha_i]_{\mu_j}\,x^{\alpha_i}.
  \end{displaymath}
\end{lemma}
\begin{proof} The expression above is obviously a well-defined map in
  $\mathbb{C}((x^{\Gamma^{\prime}}))$, due to the linear independence of the classes. A
  straightforward computation shows that it is linear and satisfies the Leibniz rule.
\end{proof}
Let us remark that the above derivation $
\DD_j$
depends on $\Gamma$ and
$\mu_1,\ldots,\mu_m$ although we do not stress this fact in the
notation. Notice also that $\DD_j$ vanishes on
$\mathbb{C}((x^{\Gamma}))$ and on $\mathbb{C}((x^{\Gamma_{j-1}}))$, for
$\Gamma_{j-1}=\langle \Gamma\cup
\{\mu_1,\ldots,\mu_{j-1}\}\rangle$.
Let us give some examples:
  Assume that $\Gamma=\{0\}$, $\mu_1=1$ and $\mu_2=\pi$, hence
  $\mathcal{D}_1(x^{2+3\,\pi})=2\,x^{2+3\,\pi}$ and
  $\mathcal{D}_2(x^{2+3\,\pi})=3\,x^{2+3\,\pi}$; now assume that
  $\Gamma=\pi\,\mathbb{Z}$, $\mu_1=2$ and $\mu_2=e$, hence
  $\mathcal{D}_1(x^{2+3\pi/2})=x^{2+3\pi/2}$,
  $\mathcal{D}_2(x^{2+3\pi/2})=0$, and
  $\mathcal{D}_2(x^{\pi+5\,e})=5\,x^{\pi+5\,e}$.
  
In what follows, we shall not usually work with the substitution $P(s(x))$ for
$s(x)\in \Omega$, but with the more useful ``full substitution'':
\begin{displaymath}
  P[s(x)]=P(s(x)+y_0,s'(x)+y_1,\ldots,s^{(n)}(x)+y_n)\in \Omega[y_0,y_1,\ldots,y_n].
\end{displaymath}
By the chain rule:
\begin{displaymath}
  \frac{\partial P}{\partial y_\kappa}[s(x)]=
  \frac{\partial }{\partial y_\kappa}(P[s(x)]),\quad
  \kappa=0,1,\ldots,n.
\end{displaymath}

For $P\in \mathbb{C}((x^{\Gamma}))[y_0,\ldots,y_n]$ and $\mu_1,\ldots,\mu_m$ and $\Gamma'$
as in Lemma~\ref{le:derivations-x^mu}, take $s(x)\in
\mathbb{C}((x^{\Gamma'}))$. Extending the derivation $
\DD_j$ to the ring
$\mathbb{C}((x^{\Gamma'}))[y_0,\ldots,y_n]$ by acting trivially on the
indeterminates $y_0,y_1,\ldots,y_n$ we get,
as the $\DD_j$ are derivations 
vanishing  on
$\mathbb{C}((x^{\Gamma}))$,
the equality:
\begin{equation}
  \label{eq:1}
  \DD_j\left( P[s(x)]\right)=
  \sum_{\kappa=0}^{n}
  \frac{\partial\,}{\partial y_\kappa}(P[s(x)])
  \,\DD_j(s^{(\kappa)}(x)).
\end{equation}

We introduce some notation useful for studying the coefficients of $P[s(x)]$. Given
$\rho=(\rho_0,\ldots,\rho_n)\in \mathbb{Z}_{\geq 0}^{n+1}$, the expression $P[s(x)]_{\rho}$ will
denote the coefficient of $y_0^{\rho_0}y_1^{\rho_1}\cdots y_n^{\rho_n}$ in $P[s(x)]$. We
shall use the basic vectors $e_0=(1,0,\ldots, 0),\ldots, e_n=(0,\ldots,0,1)$, each having
$n+1$ components (we start at $0$ because $e_{i}$ will be related to $y_i$). Examining
each term on both sides of \eqref{eq:1}, we obtain the key equality:
\begin{equation}
  \label{eq:2}
   \DD_j(P[s(x)]_{\rho})=
  \sum_{\kappa=0}^{n}(\rho_\kappa+1)
  (P[s(x)]_{\rho+e_{\kappa}})
  \,\DD_j(s^{(\kappa)}(x)).
\end{equation}

\section{Solutions and Rational Rank}\label{sec:twotheorems}
Our first  theorem follows mainly from
Lemma~\ref{le:derivations-x^mu} and Equality \eqref{eq:2}.
\begin{theorem}\label{the:main1}
  Let $P\in \Omega[y_0,\ldots,y_n]$ be a non-zero polynomial, and $s(x)\in \Omega$ a
  solution of the equation $P=0$. Then:
  \begin{equation*}
    \dim_{\mathbb{Q}} \frac{\left\langle \supp s(x) \cup  \supp P 
      \right\rangle}
    {\llangle \supp P\rrangle}
    \leq
    n .
  \end{equation*}
\end{theorem}
Thus we obtain a bound for the rational rank of solutions of equations
with coefficients Puiseux series.
\begin{corollary}
  If $s(x)\in \Omega$ is a generalized power series with
  \begin{equation*}
    \dim_{\mathbb{Q}}\frac{\langle\supp s(x)\cup \mathbb{Q}\rangle}
    {\mathbb{Q}} > n  
  \end{equation*} 
  then $s(x)$ is not the solution of a nontrivial  equation
  of order $n$ over the fraction field of the ring of Puiseux series.
\end{corollary}
\begin{corollary}
  If $s(x)\in \Omega$ is a generalized power series with
  \begin{equation*}
    \dim_{\mathbb{Q}}\langle\supp s(x)\rangle > n  
  \end{equation*} 
  then $s(x)$ is not the solution of a nontrivial equation $P=0$,
  of order $n$ with constant coefficients.
\end{corollary}
For instance $x+x^{\pi}$ is not a solution of a non-trivial equation
$P=0$ of order one and constant coefficients,
that is, with $P(y_0,y_1)\in\mathbb{C}[y_0,y_1]$. We will improve this
result in the case of autonomous first order differential equations
(see Theorem~\ref{th:AutFirstOrderODEs}).

Our second main result deals with the case when $P$ is a polynomial in
$x$ also. Notice that in this case $\langle\supp P\rangle$ is either
  $\{0\}$ if $P$ has constant
coefficients or $\mathbb{Q}$ otherwise.

\begin{theorem}\label{the:main2}
  Take $P\in \mathbb{C}[x,y_{0},\dots,y_n]$
  with $P\neq 0$ and assume
  $s(x)\in \Omega$ is a solution of
  $P=0$ whose support has maximal rational rank, that
  is
  \begin{equation*}
      \dim_{\mathbb{Q}} \frac{\left\langle \supp s(x) \cup
          \supp P\right\rangle}
      {\langle\supp P\rangle}
      = n.
  \end{equation*}
  Then $s(x)$  converges uniformly in any sector $S$ of
  sufficiently small radius with vertex at the origin and of the
  opening less than $2\,\pi$.
\end{theorem}
The following result covers Theorem 11 in~\cite{Cano2022137} and  is a
consequence of the previous Theorem for the case of first order equations
with constant coefficients:
\begin{corollary}\label{co:corTh2FirstOrderConstant}
  Let $P(y_0,y_1)$ a polynomial with constant coefficients. Let
  $s(x)\in \Omega$ be a
  solution of the equation $P=0$. Then $s(x)$ is convergent in the
  sense of Theorem~\ref{the:main2}.
  Moreover, if $s(x)$ is a Puiseux series then it is convergent in a
  neighbourhood of the origin.
\end{corollary}
\begin{proof}
  Either $s(x)$ is a constant or $\dim_\mathbb{Q}\langle \supp
  s(x)\rangle \geq 1$. Since $\langle \supp P\rangle =\{0\}$, by
  Theorem~\ref{the:main1},
  $\dim_\mathbb{Q}\langle \supp s(x)\rangle = 1$ and the result
  follows from Theorem~\ref{the:main2}.
\end{proof}
We divide the proofs of Theorems~\ref{the:main1} and~\ref{the:main2}
into a shared initial part and their respective ends.

\noindent\emph{Common part of the proof of Theorems \ref{the:main1} and \ref{the:main2}}.
Let $s(x)=\sum_{i=1}^{\infty} a_i\,x^{\alpha_i}\in \Omega$ be a solution
of the non-trivial equation $P=0$. We may
assume without loss of generality that for some $0\leq \kappa\leq n$,
$\frac{\partial P}{\partial y_\kappa}(s(x))\neq 0$: otherwise we proceed by induction on
the total degree of $P$, as $\frac{\partial P}{\partial y_{\kappa}}$ and $s(x)$ would
satisfy the hypotheses of both results. Thus, we can define:
\begin{displaymath}
  \mina=\min\{\ord_{x}
  \frac{\partial P}{\partial y_\kappa}(s(x))-\shift\,\kappa
  \mid \, 0\leq \kappa\leq
  n\},
  \,\,
  \Lambda=
  \min\{\ord_{x}
  \frac{\partial^{\rho} P}{\partial y^{\rho}}(s(x))\mid \, \rho\in
  \mathbb{Z}_{\geq 0}^{n+1}\},
\end{displaymath}
This $\lambda$ corresponds to
the one defined in~\cite{Gontsov2022579} and $\Lambda$ is just an
auxiliary constant to be used in the forthcoming proof.

Consider 
\begin{displaymath}
  m_0=
\dim_{\mathbb{Q}}\frac{\langle \supp s(x) \cup \supp P \rangle}
{\langle\supp P \rangle}
\in \mathbb{Z}_{\geq 0}\cup \{\infty\},
\end{displaymath}
which in theory could be infinity. Part of our argument will consist
in proving that it cannot be.
In the proof of Theorem~\ref{the:main1} we shall assume that $m_{0}\geq
n+1$ and argue by contradiction, whereas in the proof of
Theorem~\ref{the:main2},
we shall assume that
$m_{0}=n$, as this is
equivalent to the hypothesis.

Let $\Gamma_0=\angles{\supp P}$. For $1\leq j\leq m_0$ (if
$m_0=\infty$ the last inequality is strict),
define $\mu_j$ and $\Gamma_j$ inductively as
follows: Assume that $\Gamma_0,\Gamma_1, \ldots, \Gamma_{j-1}$ and,
$\mu_1,\mu_2,\ldots,\mu_{j-1}$ have been defined (so that for $j=1$
only $\Gamma_0$ exists). By the definition of $m_0$, there exists the
minimum $k$ of the set of
indexes $i$ such that $a_i\neq 0$ and $\alpha_i\not\in \Gamma_{j-1}$.
We set $\mu_j=\alpha_k$ and $\Gamma_{j}=\angles{\Gamma_{j-1}\cup\{\mu_j\}}$.

\newcommand{\bars}{\bar{s}}

As $P(s(x))=0$, we can choose $N$ large enough such that the truncation
$\bars(x)=\sum_{i=1}^{N}a_i\,x^{\alpha_i}$
of $s(x)$ satisfies the following property:
\begin{equation}
  \label{eq:T1}
   \ord (s(x)-\bars(x))>|\mina|+2\,n + |\Lambda|,
\end{equation}
which is possible because
either $s(x)$ has a finite number of terms or its exponents tend to
infinity. 
Notice that $\bars(x)$ could be equal to $s(x)$ if the latter has a
finite number of terms.

Once $N$ is set, there exists a finite $m$ such that 
we can rewrite $\bars(x)$ as follows:
  \begin{equation}\label{eq:r(x)}
    \bars(x) =
    \sum_{j=0}^{m}\big(
  c_{j,0}\,x^{\mu_{j,0}} + c_{j,1}\,x^{\mu_{j,1}} + \cdots + c_{j,t_j}\,x^{\mu_{j,t_j}}
  \big),
\end{equation}
where the following properties hold:
\begin{enumerate}[(i)]
\item The sequence of exponents $(\mu_{j,i})$ is increasing with respect the
  lexicographical order of their indices $(j,i)$.

\item The exponents
  $\mu_{j,i}\in \Gamma_j$, for $i=0,\ldots,t_j$ and $j=0,\ldots,m$.
  
\item All the coefficients $c_{j,i}$ are non-zero for $j=1,\ldots,m$ and
  $0\leq i\leq t_{j}$.
  
\item The exponent $\mu_{j,0}=\mu_j\not\in \Gamma_{j-1}$
  for $j=1,..,m$, that is, it adds one to
  the rational rank.
  
\end{enumerate}
By properties
(i)--(iv), the set $\Gamma=\Gamma_0$, the exponents $\mu_1,\ldots,\mu_{m}$
and $\Gamma_m$ satisfy the hypothesis of
Lemma~\ref{le:derivations-x^mu}.
In order to make $\mathbb{C}((x^{\Gamma_m}))$ a differential
field with the operator $\frac{d\,}{d\,x}$ we do the following: If
$1\not\in \Gamma_m$, we set $\Gamma'=\langle\Gamma_m\cup\{1\}\rangle$
otherwise we set $\Gamma'=\Gamma_m$.
Hence $\Gamma$ and $\mu_1,\ldots,\mu_m$, and eventually $1$,
satisty hyphotesis of Lemma~\ref{le:derivations-x^mu}
and $\mathbb{C}((x^{\Gamma'}))$ is closed under the three
operators that we use.
Any coefficient of $P[\bars(x)]$
belongs also to $\mathbb{C}((x^{\Gamma'}))$, and we may apply the derivations
$\DD_j$, $j=1,\ldots,m$, to these coefficients
(if $1\not\in \Gamma_m$, we will not use the operator
$\mathcal{D}_{m+1}$ corresponding to $1$).

Noticing that $P[\bars(x)]_{e_{k}}=\frac{\partial P}{\partial y_{\kappa}}(\bars(x))$,
Equation~\eqref{eq:2} applied to $\rho=(0,0,\ldots,0)$ gives:
\begin{equation}
  \label{eq:2_0000}
   \DD_j(P(\bars(x)))=
   \sum_{\kappa=0}^{n}
   \frac{\partial P}{\partial y_\kappa}(\bars(x))
  \,\DD_j (\bars^{(\kappa)}(x))
\end{equation}
From now on, the notation $t(x)=ax^{\mu} + \cdots $ will mean that $a\in \mathbb{C}$ ($a$
may be $0$), and $\ord_x(t(x)-a\,x^{\mu})>\mu$. By definition of $\mina$, we can write, for
$\kappa\in 0,\ldots, n$:
\[
  \frac{\partial P}{\partial
    y_\kappa}(s(x))=d_\kappa\,x^{\mina+\shift \kappa}+\cdots
\]
and we know that there is at least one $\kappa\in \{0,\ldots, n\}$ with
$d_{\kappa}\neq 0$. Using the Taylor expansion of the left hand side
of this equality and property~\eqref{eq:T1},
we also obtain
\begin{displaymath}
    \frac{\partial P}{\partial
    y_\kappa}(\bars(x))=d_\kappa\,x^{\mina+\shift\,\kappa}+\cdots.
\end{displaymath}
From the properties of the derivations
$\DD_j$, and because $[\mu_{j,0}]_{\mu_j}=1$ and
$[\mu_{j',i}]_{\mu_j}=0$ for every $(j',i)<(j,0)$, for $j=1, \ldots,m$, we have:
\begin{displaymath}
 \DD_j
 (\bars^{(\kappa)}(x))=c_{j,0}\,\dmuk{\mu_j}{\kappa}\,x^{\mu_j-\shift \kappa}+\cdots.
\end{displaymath} 
Hence, the right  hand side of Equation~\eqref{eq:2_0000} is
\begin{equation}\label{eq:proof_theorem1y2}
  \sum_{\kappa=0}^{n}
  (d_\kappa\,x^{\mina+\shift\,\kappa}+\cdots)
  (c_{j,0}\,\dmuk{\mu_j}{\kappa}\,x^{\mu_j-\shift\,\kappa}+\cdots)=
  c_{j,0}\left(\sum_{\kappa=0}^{n}d_\kappa\,\dmuk{\mu_j}{\kappa}\right)x^{\mina+\mu_j}+\cdots.
\end{equation}
This finishes the common arguments.

\begin{proof}[End of proof of Theorem \ref{the:main1}]
  In order to obtain a contradiction, we assume that
  $m_0>n$ so that $\mu_{n+1}$ is defined.
  This allows us to choose $N$
  large enough such that
  $m\geq n+1$ and 
  $\ord_x P(\bars(x)) > \mina+\mu_{n+1}$. Now $\mu_{n+1,0}=\mu_{n+1}$,
  and all this properties imply that the left hand side of
equation~\eqref{eq:2_0000} has order greater than
$\mina+\mu_{n+1,0}$.
Since $c_{j,0}\neq 0$, using~\eqref{eq:proof_theorem1y2}
we obtain the $n+1$ equalities:
\begin{equation}\label{eq:vandermonde1}
  \sum_{\kappa=0}^{n}d_\kappa\,\dmuk{\mu_j}{\kappa}=0,\quad j=1,2,\ldots,n+1,
\end{equation}
which is a square linear system with coefficient matrix
$(\dmuk{\mu_j}{\kappa})$. Because $|q|\neq 1$ and $\mu_j$ is an
increasing sequence of real numbers, then
$q^{\mu_j}\neq q^{\mu_i}$ for $i\neq j$. Thus the coefficient matrix is a 
Vandermonde
matrix in
the case of $q$-difference equations
and in the case of differential equations with the Euler operator.
In the case of differential
equations with the ordinary differential operator, the matrix
$(\dmuk{\mu_j}{\kappa})$ is reduced by elementary column operations to the
Vandermonde matrix $(\mu_{j}^{\kappa})$. Hence,
system~\eqref{eq:vandermonde1} has
the unique solution $d_0=d_1=\cdots=d_n=0$, contradicting the fact that at
least one $d_{\kappa}$ is non-zero. Thus, $m_{0}\leq n$ which gives Theorem~\ref{the:main1}.
\end{proof}

\begin{proof}[End of proof of Theorem \ref{the:main2}]
  As $m_0=n$ in this case, $\mu_n$ is defined and we can choose $N$ such that 
    $\ord_x P(\bars(x)) >   \mina+\mu_{n}$, and $m=n$. Thus
    $\mu_{n,0}=\mu_n$ is also defined. The left hand side of equation~\eqref{eq:2_0000} has order
  greater than $\mina+\mu_{n,0}$ and we obtain the system of
  equations~\eqref{eq:vandermonde1} for $j=1,2,\ldots,n$.  If $d_n=0$, this system of
  equations becomes another Vandermonde system for $d_0,d_1,\ldots,d_{n-1}$, so that
  $d_0=d_1=\cdots=d_{n-1}=d_n=0$, again contradicting the existence of
  one non-zero~$d_k$.
  As a consequence, $d_n\neq 0$, so that
  $\ord \frac{\partial P}{\partial y_n}(s(x))=
  \mina+\shift \,n$,
  and we get:%
\begin{equation}
  \label{eq:3}
  \ord \frac{\partial P}{\partial
    y_n}(s(x))-\shift\,n
  =\mina\leq
  \ord \frac{\partial P}{\partial
    y_\kappa}(s(x))-\shift\,\kappa
  ,\quad \text{for }\kappa=0,1,\ldots,n.
\end{equation}%
This means that the linearized differential
operator $\sum_{\kappa=0}^{n}\frac{\partial P}{\partial
  y_\kappa}(s(x))\, y_\kappa$ along $s(x)$ has a regular
singularity at $x=0$. If $s(x)$ is a Puisuex formal power series,
we apply the main result in~\cite{Malgrange:1989} to guarantee that
$s(x)$ is convergent in both differential cases.
If $s(x)\in \Omega$, Condition~\eqref{eq:3}
is  the hypothesis of Theorem 1 of
\cite{gontsov-goryuchkina2015}, 
which proves the Euler differential case.
To prove the ordinary differential case we only need to rewrite $P$
in terms of the Euler operator. Condition~\eqref{eq:3} becomes then the
hypothesis  of Theorem 1 of~\cite{gontsov-goryuchkina2015} again.

For the case of $q$-difference equations we need to distinguish the
cases $|q|>1$ and $|q|<1$, as in~\cite{Gontsov2022579}. If $|q|>1$,
the fact that $d_n\neq 0$ and that $\mu_j$ is an increasing sequence
of positive real numbers let us apply Theorem 1
of~\cite{Gontsov2022579} straightforwardly. If $|q|<1$, in order to
apply the same Theorem we need to show that $d_0\neq 0$. As in the
previous argument, if $d_0=0$, then~\eqref{eq:vandermonde1}, for
$j=1,\ldots,n$ becomes a Vandermonde system for $d_1,\ldots,d_n$,
whose only solution is $d_1=\cdots=d_n=0$, getting the same
contradiction. Thus $d_0\neq 0$ and we can apply
Theorem 1 of~\cite{Gontsov2022579} again to obtain the convergence. 
\end{proof}


\section{The Newton-Puiseux Polygon}\label{sec:NewtonPolygon}
In this section we give a short description of the well known method of the Newton polygon
(or Newton-Puiseux) applied to polynomial differential and 
$q$\nobreakdash-difference equations, which we shall use to
prove Theorem~\ref{the:main3}.

Let us introduce some notation:
For $\rho=(\rho_0,\ldots, \rho_n)\in \mathbb{Z}_{\geq 0}^{n+1}$ we
denote $|\rho|=\rho_0+\rho_1+\cdots+\rho_n$ and
$\omega(\rho)=\rho_1+2\,\rho_2+\cdots+n\,\rho_n$.
As before, we set $\shift=0$ in the cases of differential equations
with respect the Euler derivative or $q$-difference equations, and
$\shift=1$ in the case of differential equations with respect the
derivative~$\frac{d\,}{d\,x}$.

Fix a non-zero polynomial $P\in \Omega[y_0,\ldots,y_n]$, and write it uniquely as
\begin{equation*}
  P = \sum_{\rho \in \mathbb{Z}_{\geq 0}^{n+1}}
  \sum_{\alpha \in \supp (P)-\shift\,\omega(\rho)} P_{(\alpha,\rho)}
  \, x^{\alpha+\shift\,\omega(\rho)}y_0^{\rho_0}\cdots y_n^{\rho_n}
\end{equation*}
with $P_{(\alpha,\rho)}\in \mathbb{C}$. 
Given
$V=(\alpha, r)\in \mathbb{R}\times \mathbb{Z}_{\geq 0}$,
$P_V$ will denote 
the sum of all the terms of $P$ corresponding to the point
$V$, that is
\begin{equation*}
  P_V =P_{(\alpha,r)}=
  \sum_{(\alpha,|\rho|)=V}
  P_{(\alpha,\rho)}x^{\alpha+\shift\,\omega(\rho)}y_0^{\rho_0}\cdots y_{n}^{\rho_n}.
\end{equation*}
There is no confusion possible between $P_{(\alpha,r)}$ and
$P_{(\alpha,\rho)}$ except when $n=0$, in which case we shall abuse
the notation as the context will clarify what value we are using.
The \emph{cloud of points} $\mathcal{C}(P)$ of $P$ is the set of points in the plane:
\begin{equation*}
  \mathcal{C}(P)
  = \left\{ (\alpha, r)\in \mathbb{R}\times \mathbb{Z}_{\geq 0} : P_{(\alpha,r)}\neq 0 \right\}.
\end{equation*}
We are interested in generalized power series with exponents
in 
increasing order.
In this
setting,
the \emph{Newton-Puiseux polygon} (Newton polygon for short) of $P$, denoted
$\mathcal{N}(P)$, is the convex envelope in $\mathbb{R}^{2}$ of the set obtained by
adjoining
the half-line $\mathbb{R}_{\geq 0}\times\{0\}$
to each point in $\mathcal{C}(P)$:
\begin{equation*}
  \mathcal{N}(P) = \mathrm{conv.env.} \left\{ (\alpha,r) +
    (\mathbb{R}_{\geq 0}\times\{0\}) : (\alpha, r)\in \mathcal{C}(P)
  \right\}.
\end{equation*}
Its border is composed of a sequence of points and segments.
Given a positive number $\mu\in \mathbb{R}$, the
\emph{supporting line} of co-slope $\mu$, is the unique line
with equation $L_{\mu}(P)\equiv \mu\,r+\alpha=\alpha_0$  with
$\alpha_0$ minimum and $L_{\mu}(P)\cap \mathcal{N}(P)\neq\emptyset$.
The \emph{element of co-slope $\mu$ of
  $\mathcal{N}(P)$}, $E_{\mu}(P)$ is that
intersection $L_{\mu}(P)\cap \mathcal{N}(P)$,
which can be either a segment (called a \emph{side} of
$\mathcal{N}(P)$) or a point
(a \emph{vertex}). In
both cases we denote by $\Top(E_{\mu}(P))$ and $\Bot(E_{\mu}(P))$ the highest and lowest
points of $E_{\mu}(P)$, respectively (if $E$ is a vertex, they
coincide).

Given $V=(\alpha,r)\in \mathbb{R}\times \mathbb{Z}_{\geq 0}$,
the \emph{indicial polynomial} of $P$ at $V$ is
\begin{displaymath}
  \Psi_{(P;V)}(T)=\sum_{(\alpha,|\rho|)=V}
  P_{(\alpha,\rho)}\,T^{\langle\rho\rangle}\in \mathbb{C}[T],
\end{displaymath}
where $T^{\langle \rho\rangle}$ is equal to $T^{\omega(\rho)}$
for the Euler differential operator or the $q$-difference
operator, and
$T^{\langle \rho\rangle}=\prod_{\kappa=1}^{n}
\big(T\,(T-1)\cdots(T-\kappa+1)\big)^{\rho_\kappa}$ for
the differential operator~$\frac{d\,}{d\,x}$.

The
\emph{characteristic polynomial} of $P$ with respect to a 
co-slope $\mu\in \mathbb{R}_{\geq 0}$ is
\begin{displaymath}
\Phi_{(P;\mu)}(C)=\sum_{(\alpha,|\rho|)\in
  E_{\mu}(P)}\dmu{\mu}^{\langle\rho\rangle}\,P_{(\alpha,\rho)}\,C^{|\rho|}
=
\sum_{V\in E_{\mu}(P)}  \Psi_{(P;V)}(\dmu{\mu} )\,C^{\height(V)}
\in \mathbb{C}[C],
\end{displaymath}
where $\height(V)$ is the ordinate of $V$. 
The key Lemma of the Newton polygon process gives
the following necessary condition, which shows the importance of the characteristic
polynomial (see \cite{Cano-Fortuny-Asterisque,Cano-Fortuny-q-diff-2022} for a short proof):
\begin{lemma}\label{le:key_lemma_Newton_Polygon}
  Let $s(x)=c\,x^{\mu}+\sum_{\alpha>\mu}c_{\alpha}\,x^{\alpha}\in
  \Omega$ be a solution of $P=0$. Then
  \[\Phi_{(P;\mu)}(c)=0.\]
\end{lemma}

In fact, this lemma translates into a sequence of necessary conditions for a power series to be a solution of $P=0$. Given $s(x)=\sum_{i=0}^{\infty}c_i\,x^{\nu_i}\in \Omega$, with
$\nu_i<\nu_{i+1}$ for all $i\geq 0$, denote, for the sake of simplicity:
\begin{displaymath}
  P_0=P,\quad  \text{and}\quad
  P_{i+1}=P_i[c_i\,x^{\nu_i}]=P[c_0x^{\nu_0}+\cdots+c_i\,x^{\nu_i}],\quad i=1,2,\ldots.
\end{displaymath}
If $s(x)$ is a solution of $P=0$, then for each $i\in \mathbb{Z}_{\geq 0}$, the series
$\sum_{l=i}^{\infty} c_l\,x^{\nu_l}$ must be a solution of
$P_i=0$. Thus, we obtain the sequence of necessary initial conditions on the coefficients of $s(x)$ when it is a solution of $P=0$:
\begin{equation}
  \label{eq:4}
  \Phi_{(P_i;\nu_i)}(c_i)=0,\quad i=0,1,\ldots.
\end{equation}
We need, however, to check whether a finite power series is effectively the truncation of
a solution. To this end, we introduce the following concepts: a finite sum
$r(x) = c_0x^{\nu_0} + \cdots +c_kx^{\nu_k}$ with $c_i\in \mathbb{C}^{\ast}$
(i.e. non-zero) and $\nu_i<\nu_{i+1}$, with $\nu_i\in \mathbb{R}$ will be called a
\emph{generalized polynomial}.
Such a generalized polynomial $r(x)$ is \emph{admissible
  for $P$, or for $P=0$} (or simply \emph{admissible}) if the necessary initial
conditions~\eqref{eq:4} are fulfilled for $i=0,1,\ldots,k$. In particular, any truncation
$r(x)$ of a solution $s(x)\in \Omega$ is an admissible generalized polynomial, but the
converse is not true, even for linear  equations: the generalized polynomial
$r(x)=x$ is admissible for the differential equation
$P=2\,y_0-y_1-x+x^2$
(in terms of the Euler differential operator)
but there is no solution $s(x)\in \Omega$
of $P=0$ having $r(x)=x$ as a truncation.
Notice that the same holds
considering $P=0$ as a $q$-difference equation with $q=\sqrt{2}$ and
$r(x)=(1+\sqrt{2}/2)\,x$.

However, the converse statement holds when  $\supp r(x)$ has maximum
rational rank. This is our third main result:

\begin{theorem}\label{the:main3}
  Assume that the equation $P=0$ of order $n$
  admits a generalized admissible polynomial
  $r(x)$ with
  \begin{equation}\label{eq:Th3}
    \dim_{\mathbb{Q}}
    \frac{\left\langle \supp r(x)\cup \supp P
      \right\rangle}
{\left\langle \supp P \right\rangle }
    \geq
    n.
  \end{equation}
  Then $r(x)$ is the truncation of a solution $s(x)$ of  $P=0$,
  and the above inequality is actually an equality.
\end{theorem}

Before proving this result, we need to study the behaviour of the Newton polygon under
changes of variables of the form $y=c\,x^{\nu}+y$.
(See \cite{Cano-Fortuny-Asterisque,Cano-Fortuny-q-diff-2022}).
\begin{lemma}\label{le:behaviourNewtonChangVar}
Let $Q=P[c\,x^{\nu}]$. Then:
\begin{enumerate}
\item $\mathcal{N}(Q)$ is contained in the closed right half-plane
  defined by the supporting line $L_{\nu}(P)$.
\item Let $h$ the height of the point $\Top(E_\nu(P))$. Then
  $\mathcal{N}(P)$ and $\mathcal{N}(Q)$ are equal above $h$. In particular
  $\Top(E_\nu(P))=\Top(E_\nu(Q))$ and $L_{\nu}(P)=L_{\nu}(Q)$.
  Moreover, $P_{(\alpha,\rho)}=Q_{(\alpha,\rho)}$ for all
  $(\alpha,\rho)$ with $(\alpha,|\rho|)$ in the border of
  $\mathcal{N}(P)$ and $|\rho|\geq h$.
\item\label{property:3} The height of the point $\Bot(E_\nu(Q))$ is zero if and only if
  $\Phi_{(P;\nu)}(c)\neq 0$.
\item The following sequence of inequalities holds:
  \begin{displaymath}
     \height(\Top(E_{\nu}(P))) \geq \height(\Bot(E_{\nu}(Q))) \geq 
    \height(\Top(E_{\mu}(Q))),
  \end{displaymath}
  where $\mu>\nu$ in the last expression.
\end{enumerate}  
\end{lemma}
We have the following characterization of generalized
admissible polynomials in terms of the Newton polygon:
\begin{lemma}\label{le:caracterizacion_admisible_polynomial}
  Let $r(x)=c_0x^{\nu_0}+\cdots+c_k\,x^{\nu_k}$ with $\nu_0<\cdots<\nu_k$, and, as above,
  $P_0=P$ and $P_{i+1}=P_{i}[c_ix^{\nu_i}]$ for $i=0,\ldots,k$. Let us
  denote $Q=P[r(x)]=P_{k+1}$. Then $r(x)$ is an admissible
  generalized polynomial for $P$ if and only the bottom vertex of
  $E_{\nu_k}(Q)$ has height greater than o equal to one.
\end{lemma}
\begin{proof}
  Assume that $r(x)$ is admissible.
  Applying Lemma~\ref{le:behaviourNewtonChangVar} iteratively, we
  obtain the inequality $\height(\Bot(E_{\nu_k}(P_{k+1}))) \geq 1$.

  Conversely,
  assume that $r(x)$ is not admissible for $P$. Let $i$ be the minimum
  index for which equation~\eqref{eq:4} does not hold. By part (3) of
  Lemma~\ref{le:behaviourNewtonChangVar}, 
  $\height(\Bot(E_{\nu_i}(P_{i+1}))=0$ and  by part (4) of the same Lemma,
  \begin{displaymath}
    0=\height(\Bot(E_{\nu_i}(P_{i+1}))\geq
    \height(\Top(E_{\nu_{i+1}}(P_{i+1}))
    \geq \height(\Bot(E_{\nu_{i+1}}(P_{i+2}))\geq 0.
  \end{displaymath}
  Applying iteratively this argument we get 
 $\height(\Bot(E_{\nu_k}(P_{k+1}))=0$.
\end{proof}
The following result is also well-known but we include its proof for
the convenience of the reader.
\begin{corollary}\label{le:existencia_lado}
  Let $r(x)$ be an admissible generalized polynomial for $P=0$, and let $Q=P[r(x)]$. Then
  either $r(x)$ is already a solution of the equation $P=0$, or
  $\mathcal{N}(Q)$ has a side $E_\nu(Q)$ with $\nu>\max(\supp(r(x)))$.
\end{corollary}
\begin{proof}
By Lemma~\ref{le:caracterizacion_admisible_polynomial}, we know that
$\height(\Bot(E_{\nu_k}(Q))) \geq 1$. 
If $\mathcal{N}(Q)$ does not intersect the horizontal axis of
coordinates then the null power series is a solution 
of $Q=0$ and therefore $r(x)$ is already a solution of the equation $P=0$.
Otherwise, there is a vertex of $\mathcal{N}(Q)$ of the form $(\beta,0)$. Let
  $(\beta',0)$ be the intersection of $L_{\nu_k}(Q)$ with the
  horizontal axis.
  Since $\height(\Bot(E_{\nu_k}(Q))) \geq 1$, the point
  $(0,\beta')\not\in \mathcal{C}(Q)$, and by~(1) in 
  Lemma~\ref{le:behaviourNewtonChangVar} 
  we get  $\beta'<\beta$, which implies the
  existence of a side $E_{\nu}(Q)$ of $\mathcal{N}(Q)$ with co-slope
  $\nu>\nu_k$.
\end{proof}

The next Lemma provides several relations between $P=0$
and some of the equations
$\frac{\partial P}{\partial y_i}=0$,
which will be useful in the proof of the main
Proposition of this section.

\begin{lemma}\label{le:multiple-root}
  Let $r(x)$ be an admissible generalized polynomial
  for $P$, and set  $Q=P[r(x)]$.
  Let $\nu> \max( \supp r(x))$ and assume 
  there exists a point 
  $V=(\alpha, r)\in E_\nu(Q)\cap \mathcal{C}(Q)$, with $r\geq 2$.
Take $\rho$ with $|\rho|=r$ such that 
$Q_{(\alpha,\rho)}\neq 0$, which exists because $V\in \mathcal{C}(Q)$.
  Let $i$ be an index $0\leq i\leq n$
   such that  $\rho_i\geq 1$, and set
  $\overline{P}=\frac{\partial  P}{\partial y_i}$ and
  $\overline{Q}=\overline{P}[r(x)]$.
  
   Then  $r(x)$ is an
  admissible generalized polynomial for
  $\overline{P}$, and 
  the point $\overline{V}=(\alpha, r-1)$ belongs to
  $E_{\nu}(\overline{Q})\cap \mathcal{C}(\overline{Q})$.
  Moreover, if $V$ is a vertex of $\mathcal{N}(Q)$ then $\overline{V}$ is a
  vertex of $\mathcal{N}(\overline{Q})$.
\end{lemma}
\begin{proof}
  The cloud of points of $\overline{P}$ is a subset of the image of the points of
  $\mathcal{C}(P)$ with height at least $1$ under the map
  $\Delta(\alpha, r)= (\alpha, r-1)$.
  By the chain rule $\overline{Q}=\frac{\partial Q}{\partial y_i}$,
  hence $\mathcal{C}(\overline{Q})$ is a subset of
the image of the points of $\mathcal{C}(Q)$ with height at least one
under $\Delta$.
  Moreover
  $\overline{Q}_{(\alpha,\rho-e_i)}=\rho_i\,Q_{(\alpha,\rho)}$, so that
   the point
   $\overline{V}=(\alpha, r-1)$ belongs to
   $\mathcal{C}(\overline{Q})\cap \Delta(L_\nu(Q))$.
   As $\Delta$ is an affine translation, the existence of
   $\overline{V}$ implies that  $\Delta(L_\nu(Q))$ is the
   supporting line of co-slope $\nu$ of $\mathcal{C}(\overline{Q})$.
  Therefore $\overline{V}\in E_{\nu}(\overline{Q})$.

  As $\height \overline{V}=r-1\geq 1$, and $\overline{V}\in E_{\nu}(\overline{Q})$,
  then $\height(\Top(E_{\nu}(\overline{Q})))\geq 1$.
  Let  $\nu_k=\max\supp r(x)$.
  Since $\nu>\nu_k$, by (4) of Lemma
  \ref{le:behaviourNewtonChangVar}
  we get $\height
  \Bot(E_{\nu_k}(\overline{Q}))\geq \Top(E_{\nu}(\overline{Q}))\geq 1$. 
   Hence by Lemma~\ref{le:caracterizacion_admisible_polynomial}, the
   generalized polynomial $r(x)$
   is admissible for $\overline{P}$.

 Assume now that $V$ is a vertex of $\mathcal{N}(Q)$, so that there
 exists $\nu'>\max\supp r(x)$ with $\nu\neq \nu'$ and $V\in E_\nu(Q)\cap
 E_{\nu'}(Q)$. We conclude that  $\overline{V}\in  E_\nu(\overline{Q})\cap
 E_{\nu'}(\overline{Q})$ which proves that $\overline{V} $ is a vertex
 of~$\mathcal{N}(\overline{Q})$. 
\end{proof}

\subsection{Admissible polynomials with maximal rational rank and
  characte\-ristic polynomials}

The goal of this subsection is to prove that
if $r(x)$ is an
admissible generalized polynomial for the equation $P=0$
which has maximum rational rank 
then 
the characteristic
polynomials of the relevant sides of the Newton polygon of $P[r(x)]$
have always non-zero roots. That is, after the last exponent of
$r(x)$, the equation $P[r(x)]=0$ behaves as an algebraic curve
and $r(x)$ can always be completed to a solution of $P=0$.

Throughout this subsection,  $r(x)$ will be denote an admissible
generalized polynomial for $P=0$ such that
\begin{equation}\label{eq:hipothesis.Th3}
  \dim_{\mathbb{Q}}
  \frac{\langle \supp P\cup \supp r(x)\rangle}
  {\langle \supp P \rangle }
= m\geq n, 
\end{equation}
where $n$ is the order of $P=0$.
Let  $\Gamma=\langle \supp P\rangle$ and
write $r(x)$ as in
Equation~\eqref{eq:r(x)}, satisfying properties (i)--(iv). Thus, we have:
\begin{equation}\label{eq:r(x)-2}
    r(x) =
    \sum_{j=0}^{m}\big(
  c_{j,0}\,x^{\mu_{j,0}} + c_{j,1}\,x^{\mu_{j,1}} + \cdots + c_{j,t_j}\,x^{\mu_{j,t_j}}
  \big),
\end{equation}
and we set, from now on, $Q=P[r(x)]$.

As in the previous section, we denote $\mu_j=\mu_{j,0}$ and
either
$\Gamma'=\langle\Gamma\cup \{\mu_1,\ldots,\mu_m\}\rangle$
or $\Gamma'=\langle\Gamma\cup \{\mu_1,\ldots,\mu_m,1\}\rangle$
according as $1\in \langle\Gamma\cup \{\mu_1,\ldots,\mu_m\}$ or not.
The derivations
$\DD_j$, $j=1,\ldots,m$, on $\mathbb{C}((x^{\Gamma'}))$ are defined as in Lemma.
Notice that
$Q\in \mathbb{C}((x^{\Gamma'}))[y_0,\ldots,y_n]$. Given
$\rho=(\rho_0,\rho_1,\ldots,\rho_n)\in \mathbb{Z}_{\geq 0}^{n+1}$ and $\alpha\in \Gamma'$, we write
$Q_\rho\in \mathbb{C}((x^{\Gamma'}))$ for the coefficient of
$y_0^{\rho_0}y_1^{\rho_1}\cdots y_n^{\rho_n}$ in $Q$, and
$Q_{(\alpha,\rho)}\in \mathbb{C}$ for the coefficient of
$x^{\alpha+\shift\,\omega(\rho)}$ in
$Q_{\rho}$. Recall that $e_i$ is a vector with $n+1$ components, all zero except the
$i+1-th$, which is $1$.

The existence of the admissible generalized polynomial $r(x)$ generates conditions on the
coefficients of $Q$ at elements of co-slope greater than $\max\supp
r(x) =\mu_{m,t_m}$.

\begin{lemma}\label{le:sum-zero}
  Let $\nu>\max\supp r(x)=\mu_{m,t_m}$ and consider the element $E_{\nu}(Q)$
  of
  co-slope $\nu$ of $\mathcal{N}(Q)$. Assume there is a point
  $V=(\alpha,r)\in E_{\nu}(Q)$, with $r=\height(V)\geq 1$.  
  Then for all
  $\rho=(\rho_0,\rho_1,\ldots,\rho_n)\in \mathbb{Z}_{\geq 0}^{n+1}$ with $|\rho|=r-1$, and
  for all $j=1,2,\ldots,m$, the following equality holds:
    \begin{equation}\label{eq:main}
      \sum_{\kappa=0}^{n}(\rho_\kappa+1)\,
      \dmuk{\mu_{j}}{\kappa}\,Q_{(\alpha,\rho+e_\kappa)}
      =0.
    \end{equation}
  \end{lemma}
  \begin{remark} 
    Notice that  each
    $\rho=(\rho_0,\ldots,\rho_n)$ with $|\rho|=r-1$ gives rise to $m$
    conditions  \eqref{eq:main} on the coefficients 
of $Q_V$.
  \end{remark}

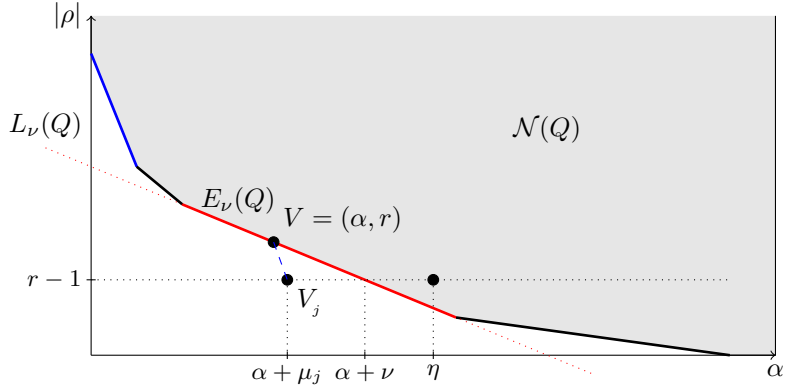
\begin{figure}
  \centering
  \begin{tikzpicture}[x=0.6cm,y=0.5cm]
     \draw[fill=black!10, draw opacity=0.1] (0,9) -- (0,8) -- (1,5) -- (2,4) -- (8,1) --
     (14,0) -- (15,0) -- (15,9);
     \draw (10,6) node {$\mathcal{N}(Q)$};
    \draw[->] (0,0) -- (15,0) node[anchor=north]{$\alpha$};
    \draw[->] (0,0) -- (0,9) node[anchor=east]{$|\rho|$};
    \draw[blue,line width=1pt] (0,8) -- (1,5);
    \draw[line width=1pt] (1,5) -- (2,4);
    \draw [red,line width=1pt](2,4) -- (8,1);
    \draw [black,line width=1pt](8,1) -- (14,0);
    \draw[red,dotted] (-1,5.5) -- (11,-0.5);
    \draw (-1,5.5) node[anchor= south] {$L_\nu(Q)$};
    \draw (2.2, 3.5) node[anchor=south west]{$E_{\nu}(Q)$};
    \draw[fill=black] (4,3) circle[radius=2pt]
    node[anchor=south west]{$V=(\alpha,r)$};
    \draw[fill=black] (4.3,2) circle[fill=black,radius=2pt]
    node[anchor=north west]{$V_{_{j}}$};
     \draw[blue, dashed] (4,3) -- (4.3,2);
     \draw[dotted] (0,2) -- (14,2);
     \draw (0,2) node[anchor = east]{\small$r-1$};
     \draw (-0.1,2)--(0.1,2);
      \draw [dotted](4.3,2) -- (4.3,0); 
      \draw (4.3,0) node[anchor = north]{\small$\alpha+\mu_j$};
      \draw (4.3,0.1) -- (4.3,-0.1);
    \draw [dotted](6,2) -- (6,0);
    \draw (6,0) node[anchor = north]{\small$\alpha+\nu$};
    \draw (6,0.1) -- (6,-0.1);
    \draw[fill=black] (7.5,2) circle[radius=2pt];
     \draw [dotted](7.5,2) -- (7.5,0);
     \draw (7.5,0) node[anchor = north]{\small$\eta$};
     \draw (7.5,0.1) -- (7.5,-0.1);
     
  \end{tikzpicture}
  \caption{Newton Polygon of $Q=P[r(x)]$ and some elements that
    appear in the proof of Lemma~\ref{le:sum-zero}.
  } 
  \label{fig:argument-order-1}
\end{figure}

\begin{proof}
   Fix for the all this proof a $\rho\in \mathbb{Z}_{\geq 0}^{n+1}$ with $|\rho|=r-1$.
  Let $V_j=(\alpha+\mu_j,r-1)$ for $j=1\,\ldots, m$. Notice that
  $V_j\not\in \mathcal{N}(Q)$ because $V\in E_{\nu}(Q)$ and $\nu>\mu_{m,t_m}\geq
  \mu_j$. Thus, $Q_{(\alpha+\mu_j,\rho)}=0$.

  Since $\mu_1,\ldots,\mu_m$ and $\Gamma$ fulfill the hypothesis of
  Lemma~\ref{le:derivations-x^mu}, we may apply the derivations
  $\DD_j$ to $Q$, and
  for each
  $j=1,\ldots, m$,
  Equation~\eqref{eq:2} becomes
  \begin{equation}
    \label{eq:5}
   \DD_j(Q_{\rho})=
  \sum_{\kappa=0}^{n}(\rho_\kappa+1)
  (Q_{\rho+e_{\kappa}})
  \,\DD_j(r^{(\kappa)}(x)).
\end{equation}
 Since $V=(\alpha,r)$ is on the border of $\mathcal{N}(Q)$ and
$|\rho+e_{\kappa}|=r$, we have:
\begin{displaymath}
  Q_{\rho+e_\kappa}=
  Q_{(\alpha,\rho+e_\kappa)}x^{\alpha+\shift\,\omega(\rho+e_\kappa)}+\cdots
  =
  Q_{(\alpha,\rho+e_\kappa)}x^{\alpha+\shift\,\omega(\rho)+\shift\,\kappa}+\cdots
\end{displaymath}
where the last equality holds because $\omega(\rho+e_\kappa)
=\omega(\rho)+\kappa$.
Because of Properties (i)--(iv) in the expression
of $r(x)$ the following equality holds:
\begin{displaymath}
 \DD_j(r^{(\kappa)}(x))=
  c_{j,0}\,\dmuk{\mu_j}{\kappa}\,x^{\mu_j-\shift\,\kappa}+\cdots,
\end{displaymath}%
where recall that $c_{j,0}\neq 0$ for each $j=1,\ldots,m$.
Hence, the right hand side of Equation~\eqref{eq:5} is equal to
\begin{equation}\label{eq:EcuacionClave2}
  c_{j,0}\left(
     \sum_{\kappa=0}^{n}(\rho_\kappa+1)\,
     \dmuk{\mu_{j}}{\kappa}\,Q_{(\alpha,\rho+e_\kappa)}
   \right)
   x^{\alpha+\mu_j+\shift\,\omega(\rho)}+\cdots,
\end{equation}%
and the left hand side of Equation~\eqref{eq:5} has, by definition of
$\DD_j$, order at least
$\ord_{x}Q_{\rho}$. Thus, if $Q_\rho=0$ we are done.
So we assume $Q_\rho\neq 0$ from now on.
Set then $\ord_{x}Q_{\rho}=\eta+\shift\,\omega(\rho)$.

It is enough to prove
that $\eta>\alpha+\mu_{j}$ because this implies that the first
term of~\eqref{eq:EcuacionClave2} is zero
which combined with~\eqref{eq:5} gives Equation~\eqref{eq:main}. 
By definition of $L_\nu(Q)$, any
point in $\mathcal{C}(Q)$ belongs to the closed right halfplane defined by
$L_{\nu}(Q)$. Since $V=(\alpha,r)\in E_\nu(Q)\subset L_{\nu}(Q)$,
the point $(\alpha+\nu,r-1)\in L_{\nu}(Q)$ because $L_{\nu}(Q)$ is a
line with co-slope $\nu$. Therefore,
any point with ordinate $r-1$ in $\mathcal{N}(Q)$ has abscissa greater
than or equal to $\alpha+\nu$.
As $Q_\rho\neq 0$ then  $Q_{(\eta+\shift\,\omega(\rho),\rho)}\neq 0$ so that
$(\eta,r-1)\in \mathcal{C}(Q)$ and hence $\eta\geq \alpha+\nu>\alpha+\mu_{j}$.  
\end{proof}

The key consequence of Lemma \ref{le:sum-zero} is the following
result.

\begin{proposition}\label{prop:main}
  Let
  $\nu>\max\supp r(x)$ and  $V=(\alpha, r)\in E_{\nu}(Q)$ with
  $Q_V\neq 0$. The indicial polynomial $\Psi_{(Q;V)}(T)$ is a power
  of the polynomial $(T-\dmu{\mu_1})\cdots(T-\dmu{\mu_n})$ up to a non-zero
  constant $c$:
  \[
    \Psi_{(Q;V)}(T) = c \,\bigg(\prod_{j=1}^n(T-\dmu{\mu_j})\bigg)^{r}.
  \]
  As a consequence, $\Psi_{(Q;V)}(\dmu{\nu})\neq 0$ because
  $\nu>\max\supp r(x)>\mu_j$ for all $j=1,\ldots,n$, so that if $E_{\nu}(Q)$ is a
  side of $\mathcal{N}(Q)$, then the characteristic polynomial
  $\Phi_{(Q,\nu)}(C)$ has at least one 
  non-zero root.
\end{proposition}
\begin{proof}
  We proceed by induction on $r$. The case $r=0$ is trivial because
  $\Psi_{(Q;V)}(T)=Q_{(\alpha,(0,\ldots,0))}$ is a non-zero constant
  under our hypothesis $Q_V\neq 0$.

  We need to prove the case
  $r=1$ because our induction argument requires $r\geq 2$.

  Assume $r=1$.  The indicial polynomial of $Q$
  at $V=(\alpha,1)$ is, by definition,
\begin{displaymath}
  \Psi_{(Q;V)}(T)=\sum_{k=0}^{n} Q_{(\alpha, e_{\kappa})}\,
  T^{\langle e_\kappa\rangle}.
\end{displaymath}
Applying Lemma~\ref{le:sum-zero} to $\rho=(0,\ldots,0)$ we get
$\Psi_{(Q;V)}(\dmu{\mu_j})=0$
for $j=1,\ldots,n$ (i.e. each $\dmu{\mu_{j}}$ is a root of $\Psi_{(Q;V)}(T)$). Since the
polynomial $\Psi_{(Q;V)}(T)$ has degree at most $n$, we must have:
\begin{displaymath}
  \Psi_{(Q;V)}(T)=\sum_{k=0}^{n}
  Q_{(\alpha,e_{\kappa})} \, T^{\langle e_\kappa\rangle} =
   Q_{(\alpha,e_{n})}\, \prod_{j=1}^{n}(T-\dmu{\mu_j}),
 \end{displaymath}
 and we need to show that $Q_{(\alpha,e_n)}\neq 0$. If it were not so,
 then $\Psi_{(Q;V)}(T)=0$ and all
 the
 coefficients $Q_{(\alpha,e_{\kappa})}$ should be zero for $\kappa=0,\ldots,n $, which
 contradicts the fact that $Q_V\neq 0$. This finishes case $r=1$.


 Assume the result holds for $r-1$ and consider the general case.  Let $V=(\alpha,r)$,
 with $r\geq 2$ and all other notations as before.  For any
 $\rho\in \mathbb{Z}_{\geq 0} ^{n+1}$ with $|\rho|=r-1$, Equation~\eqref{eq:main} for
 $j=1,\ldots,n$ implies the following equality of polynomials in the
 variable $\mu$ of degree at most $n$:
\begin{equation}\label{eq:prop-ind-r-0}
 \sum_{\kappa=0}^n (\rho_{\kappa}+1) \, T^{\langle e_\kappa\rangle} \,
 Q_{(\alpha, \rho+e_{\kappa})} =
 (\rho_n+1)\,Q_{(\alpha,\rho+e_n)}\,\prod_{j=1}^{n}(T-\dmu{\mu_j}),
\end{equation}
because
$\dmu{\mu_j}$, for $j=1,\ldots,n$, are different  roots of the left hand side.

We are going to prove that $Q_{(\alpha,(0,\ldots,0,r))}\neq 0$ and
apply Lemma~\ref{le:multiple-root} to use the induction hypothesis.
Arguing by contradiction  assume  $Q_{(\alpha,(0,\ldots,0,r))}=0$. Take
$\gamma\in\mathbb{Z}_{\geq 0} ^{n+1}$ with $|\gamma|=r$ maximal with respect the
reverse lexicographical order such that
$Q_{(\alpha,\gamma)}\neq 0$. This $\gamma$ exists because
$Q_V\neq 0$, and $\gamma\neq (0,0,\ldots,r)$ by the assumption.
Let $\rho=\gamma-e_\kappa$ for some $\kappa<n$ with $\gamma_{\kappa}>0$,
which exists too because $\gamma_n<r$.
Equation~\eqref{eq:prop-ind-r-0} applied to this $\rho$ implies that
if $Q_{(\alpha,\rho+e_n)}=0$, then the left hand side
of~\eqref{eq:prop-ind-r-0} is the null polynomial so that
 $Q_{(\alpha,\gamma)}=0$ because
$\gamma=\rho+e_{\kappa}$.
Hence
$Q_{(\alpha,\rho+e_n)}\neq 0$. This contradicts the definition of
$\gamma$. Thus our assumption is false and 
$Q_{(\alpha,(0,\ldots,0,r))}\neq 0$.

Equation~\eqref{eq:prop-ind-r-0} is important because it will allow us
to show that $\Psi_{(Q;V)}(T)$ is a polynomial divisible by 
$\prod_{j=1}^{n}(T-\dmu{\mu_j})$ from which we shall be able to apply the
induction hypothesis.

Multiplying both sides of
Equation~\eqref{eq:prop-ind-r-0} by
$T^{\langle\rho\rangle}$,
as
$T^{\langle \rho\rangle} T^{\langle e_\kappa\rangle}=
T^{\langle \rho+e_\kappa\rangle}$,
we get
\begin{equation}\label{eq:prop_ind_r_w(rho)}
   \sum_{\kappa=0}^n (\rho_{\kappa}+1) \, T^{\langle\rho+e_\kappa\rangle} \,
 Q_{(\alpha, \rho+e_{\kappa})} =
 (\rho_n+1)\,Q_{(\alpha,\rho+e_n)}\,
 T^{\langle\rho\rangle}\,\prod_{j=1}^{n}(T-\dmu{\mu_j}).
\end{equation}
The sum of all the left hand sides of~\eqref{eq:prop_ind_r_w(rho)}
for $\rho$ with $|\rho|=r-1$ gives:
\begin{displaymath}
  \sum_{|\rho|=r-1}\sum_{\kappa=0}^n (\rho_{\kappa}+1) \,
  T^{\langle\rho+e_\kappa\rangle} \,
    Q_{(\alpha, \rho+e_{\kappa})}
=
r\,\sum_{|\gamma|=r}T^{\langle \gamma\rangle}\,Q_{(\alpha,\gamma)}=r\,\Psi_{(Q;V)}(T),
\end{displaymath}
where the second equality is the definition of the indicial
polynomial, and the first one follows from collapsing the preimages of
the map $(\rho,\kappa)\mapsto \gamma=\rho+e_\kappa$. As a
consequence, summing all the corresponding right hand sides
of~\eqref{eq:prop_ind_r_w(rho)} we get 
\begin{equation}\label{eq:prop-ind-r-2}
  r\, \Psi_{(Q;V)}(T)=\prod_{j=1}^{n}(T-\dmu{\mu_j}) \sum_{|\rho|=r-1}
  (\rho_n+1)\,Q_{(\alpha,\rho+e_n)}\,T^{\langle\rho\rangle}.
\end{equation}
We are now going to apply the induction hypothesis in order to prove that the summation on
the right hand side is, up to a non-constant factor, an $r-1$ power of
$\prod_{j=1}^{n}(T-\dmu{\mu_j})$. To this end, recall that
$Q_{(\alpha,(0,0,\ldots,r))}\neq 0$ and consider
$\overline{P}=\frac{\partial P}{\partial y_n}$. By the chain rule,
$\overline{Q}=\overline{P}[r(x)]=\frac{\partial Q}{\partial y_n}$, and in particular,
$\overline{Q}_{(\alpha,(0,0,\ldots,r-1)}=r\,Q_{(\alpha,(0,0,\ldots,r))}\neq
0$. Lemma \ref{le:multiple-root}
guarantees that $r(x)$ is an admissible generalized polynomial for
$\overline{P}$.
Let $\overline{V}=(\alpha, r-1)$.
We know that then $\overline{Q}_{\overline{V}}\neq 0$, and by
Lemma~\ref{le:multiple-root} again,
$\overline{V}$ belongs to an element $E_{\nu}(\overline{Q})$ of
co-slope $\nu$.
These properties allow us to apply the induction hypothesis to
$\overline{P}$, $r(x)$ and $\overline{V}$ which means that
the indicial polynomial $\Psi_{(\overline{Q},\overline{V})}(T)$ 
is, up to a non-zero constant, an
$r-1$-th power of $\prod_{j=1}^{n}(T-\dmu{\mu_j})$. Thus, on one hand, by
definition and because $\overline{Q}=\frac{\partial Q}{\partial y_n}$
we have 
\begin{equation}\label{eq:ind1}
  \Psi_{(\overline{Q},\overline{V})}(T)=
  \sum_{|\rho|=r-1}\overline{Q}_{(\alpha,\rho)}\,T^{\langle \rho\rangle }=
  \sum_{|\rho|=r-1}(\rho_n+1)\,Q_{(\alpha,\rho+e_n)}\,T^{\langle \rho\rangle },
\end{equation}
and, on the other, by the induction hypothesis, 
there is a non-zero constant $c$ with
\begin{equation}\label{eq:ind2}
  \Psi_{(\overline{Q},\overline{V})}(T) = 
c \,\bigg(\prod_{j=1}^{n}(T-\dmu{\mu_j})\bigg)^{r-1}.
\end{equation}
The proof of the first statement in the Proposition finishes by
connecting \eqref{eq:ind1} and 
\eqref{eq:ind2}, and inserting the result into \eqref{eq:prop-ind-r-2}.

It only remains to prove that the characteristic polynomial $\Phi_{(Q,\nu)}(C)$
has at least one non-zero root. By definition:
\begin{displaymath}
  \Phi_{(Q,\nu)}(C)=\sum_{W\in
    E_{\nu}(Q)}\Psi_{(Q;W)}(\dmu{\nu})\,C^{\height(W)}.
  \end{displaymath}
We have already proved that for all $W\in E_{\nu}(Q)\cap \mathcal{C}(Q)$,
$\Psi_{(Q;W)}(\dmu{\nu})\neq 0$, because $\nu>\mu_j$, for all $j=1,\ldots,n$.
Since $E_{\nu}(Q)$ is a side, it has at least two vertices, whence
$\Psi_{(Q,\nu)}(C)$ has at least two non-zero monomials,
and thus it has at least one non-zero root. \end{proof}

\begin{proof}[Proof of Theorem \ref{the:main3}] 
Let $Q=P[r(x)]$. Consider the following procedure, analogue to the
Newton-Puiseux construction for algebraic curves: 

\begin{procedure} \textsc{Completion of $r(x)$ to a solution of $P=0$}.
  \begin{enumerate}          
  \item[Setup.] \textsc{Set} $r_0(x)=r(x)$, $Q_1=Q$,
    $\eta_0=\max(\supp (r(x))$, and $i=1$.
  \item[1.] \textsc{If} $y=0$ is a solution of $Q_{i}=0$ \textsc{Then Return}
    $s(x)=r_{i-1}(x)$ \textsc{Else}
  \item[2.] \textsc{Choose} a side of $\mathcal{N}(Q_{i})$ with co-slope
    $\eta_i>\eta_{i-1}$ and 
  \item[3.] \hspace{3ex} a non-zero root $d_i$ of the characteristic
    polynomial $\Phi_{(Q_{i},\eta_i)}(C)$.
  \item[4.] \textsc{Set} 
     $r_i(x)=r_{i-1}(x)+d_i\,x^{\eta_i}$ and   
      $Q_{i+1}=Q_{i}[d_i\,x^{\eta_i}]$.
    \item[5.] \textsc{Set} $i=i+1$ and \textsc{Goto} 1.
    \item[6.]  \textsc{Return} 
    $s(x)=r(x)+\sum_{i=1}^{\infty}d_i\,x^{\eta_i}$.
  \end{enumerate} 
\end{procedure}  

 We shall show that in each iteration, the main loop of  
the above procedure   
can be performed, that the output $s(x)$ belongs to $\Omega$, and it
is a solution of the equation~$P=0$. From this follows,
by Theorem~\ref{the:main1}, that $m\leq n$, hence $m=n$.

Let us prove that each step can be performed.
Assume $r_{i-1}(x)$, $Q_i$, $\eta_{i-1}$ are defined for some $i\geq
1$ and $r_{i-1}(x)$ is admissible for $P=0$. 
If the condition of line 1 holds the procedure finished. Otherwise, as
$r_{i-1}(x)$ is admissible for $P=0$ there exists
$\eta_{i}>\eta_{i-1}$ by Corollary~\ref{le:existencia_lado} such that
$\mathcal{N}(Q_i)$ has a side of co-slope $\eta_i$. This allows
line 2  of Procedure 1 to be performed.
Proposition~\ref{prop:main} guarantees the existence of a
non-zero root $d_i$ of the characteristic 
polynomial~$\Phi_{(Q_{i},\eta_i)}(C)$ which gives line 3.

We have proved that the above method either returns a generalized
polynomial $r_{i-1}(x)$ which is a solution of $P=0$ or, after
infinitely many steps, it returns a formal power series
$s(x)=r(x)+\sum_{i=1}^{\infty}d_i\,x^{\eta_i}$.  
From the definition of $\eta_0$ and the fact that $\eta_0<\eta_j$,
it is obvious that $r(x)$ is a truncation of $s(x)$
but it remains to prove that  $s(x)\in \Omega$ and that
$s(x)$ is a solution of $P=0$.
These facts follow from the following Proposition, which can be found
in the literature under slightly different hypotheses:
In \cite{Grigoriev-Singer}
the equation $P=0$ is differential and has  coefficients in the field
$\mathbb{C}((x^\mathbb{Z}))$, in \cite{CanoJ} the series $s(x)$ has rational
exponents, and in \cite{Cano-Fortuny-q-diff-2022} $P=0$ is a
$q$-difference equation whose coefficients are grid-based series. 
The arguments presented in those proofs
are equally valid in our setting with minimal modification.
For the convenience of the reader we sketch the proof of the specific
statement we need, thus ending the proof of Theorem~\ref{the:main3}.

\end{proof}

\begin{proposition}\label{pro:stabilization}  
  Let $P\in\Omega[y_0,y_1,\ldots,y_n]$ and 
  $s(x)=\sum_{i=0}^{\infty}c_i\,x^{\nu_i}$ be a series with
  $\nu_0<\nu_1<\dots$ such that $s(x)$
  satisfies the necessary initial conditions~\eqref{eq:4} for any
  $i\geq 0$.  Then $s(x)\in \Omega$ and it is solution of the
  equation $P=0$.
\end{proposition}

\begin{proof}
First, we are going to define a special point which we call the \textit{pivot point} 
of $P$ along the series
$\sum_{i=0}^{\infty}c_i\,x^{\nu_i}$ and will be central to our
arguments.

As previously, we denote $P_0=P$ and $P_{i+1}=P_i[c_i\,x^{\nu_i}]$. 
By
Lemma~\ref{le:behaviourNewtonChangVar} we know that
$\height \Top(E_{\nu_i}(P_i)) \geq \height
\Top(E_{\nu_{i+1}}(P_{i+1}))\geq 1$, for any $i\geq 0$. Hence there exists an index $i_0$
such that this sequence of heights stabilizes, and therefore
the points $\Top(E_{\nu_i}(P_i))$ stabilize too: there is a point
$V=(\alpha,r)$ such that
$\Top(E_{\nu_{i_0}}(P_{i_0})= \Top(E_{\nu_i}(P_i))=V $ for all $i\geq
i_0$. This $V$ is the \emph{pivot point} of $P$ along~$s(x)$.

We know that $r\geq 1$ by Lemma~\ref{le:behaviourNewtonChangVar}.
Notice that for any $k\geq i_0$, setting
$\alpha_k=\ord_x P(\sum_{i=0}^{k-1}c_i\,x^{\nu_i})$ one
has $\alpha_k\geq \alpha+r\,\nu_k$
because the point $(\alpha_k,0)\in \mathcal{C}(P_k)$ and the
supporting line
$L_{\nu_k}(P_k)$ passes through $V$.  Hence, in order to prove the
Proposition it is
enough to show that $s(x)\in \Omega$, 
that is $\lim \nu_i=\infty$, because then 
automatically $\lim \alpha_k=\infty$ which makes $s(x)$ a solution of~$P=0$.

We may assume that $r=1$, because if $r\geq 2$ we can apply
$r-1$ times
Lemma~\ref{le:multiple-root}
 to find some derivative
$P'=\frac{\partial^{r-1} P}{\partial y_0^{j_0}\dots\partial y_n^{j_n}}$ of $P$, such
that, the height of the pivot point of $P'$ along $s(x)$ is one.
Thus from now on we assume that $r=1$.

We may also assume that $\alpha=0$ so that $V=(0,1)$.
Let us sketch why (a detailed argument can be found
in~\cite{Cano-Fortuny-q-diff-2022}). 
Performing the change of variable $\bar{y}=x^{\nu_{i_0}}\,y$
on $P=0$,
and multiplying
the resulting equation by $x^{-\beta}$, with $\beta=\alpha+\nu_{i_0}$,
we obtain a new equation $\overline{P}=0$ whose
pivot point  along $\bar{s}(x)=x^{-\nu_{i_0}}\,s(x)$
is  $(0,1)$ and such that the coefficients of
$\overline{P}_{i_0}$ are elements of $\Omega$ with
non-negative order in $x$.
Using $\overline{P}_{i_0}$ instead of $P$, 
we may assume that $P=0$ is in \emph{quasi-solved form}:
$\mathcal{N}(P)$ is contained in the first quadrant,
and in particular
the coefficients of $P$ have
non-negative order in $x$,
and
$V=(0,1)$ is a vertex of $\mathcal{N}(P)$, and
$\nu_0\geq 0$.

We first prove the Proposition assuming that $P$ is in quasi-solved form
and with the additional hypothesis that its coefficients  are
grid-based, that is: there exists a finitely generated semigroup $G$
of $\mathbb{R}_{\geq 0}$ such that $\supp P\subset G$.
In this case we can prove that $s(x)$ is also grid-based.
The general
case will follow from this one.

By
Lemma~\ref{le:behaviourNewtonChangVar}, the point $V=(0,1)$ is a vertex
of $\mathcal{N}(P_i)$ for all $i\geq 0$. Also by
Lemma~\ref{le:behaviourNewtonChangVar},  the 
equality $P_{(0,e_\kappa)}=({P_i})_{(0,e_\kappa)}$ holds for 
all $i\geq 0$ and $\kappa=0,1,\ldots,n$. 
Hence the indicial polynomials 
$\Psi_{(P;V)}(T)$ and $\Psi_{(P_i;V)}(T)$ are equal for all 
$i\geq 0$. 
Let $\Sigma=\{\mu\in \mathbb{R}_{\geq 0}\mid 
\Psi_{(P;V)}(\dmu{\mu})=0\}$.
This set is finite because the set of roots of the polynomial
$\Psi_{(P;V)}(T)=\sum_{\kappa=0}^{n}P_{(0,e_\kappa)}\,T^{\kappa}$
is finite (it is not identically null since $(0,1)$ is a vertex),
and $|q|\neq 1$ in the case of $q$-difference equations.
Let $G'$ be the semigroup generated by $G$ and $\Sigma$,
which is finitely generated.

Let us prove inductively that  $\supp P_i\subset G'$ and
that $\nu_i\in G'$. The case $i=0$ holds by hypothesis.
Let  $i\geq 0$ and assume that $\supp P_i\subset G'$; let
us first prove that $\nu_i\in G'$. 
By hypothesis, $\Phi_{(P_i;\nu_i)}(c_i)=0$.
If the element $E_{\nu_i}(P_i)$ is a
side then its vertices are $(0,1)$ and $(\alpha_i,0)$, so that
$\nu_i=\alpha_i\in \supp P_{i}\subset G'$.
Otherwise $E_{\nu_i}(P_i)$ is the vertex $\{V\}$, which implies that
$\Phi_{(P_i;\nu_i)}(c_i)=\Psi_{(P_i,V)}(\dmu{\nu_i})c_i$,
so that $\nu_i\in \Sigma$.
In both cases, $\nu_i\in G'$.
Since $P_{i+1}=P_i[c_i\,x^{\nu_i}]$ then
$\supp P_{i+1}\subset G'$.
This shows that $\supp s(x)\subset G'$ which implies that $\lim
\nu_i=\infty$ as desired.

It only remains to prove the Proposition
when $P$ has coefficients  in $\Omega$, not necessarily grid-based. 
Take
$N>0$ and denote by
$R$ the polynomial obtained by truncating in the variable $x$ the coefficients of
$P$ up to orden $N$.
Define as usually $R_0=R$ and $R_{i+1}=R_i[c_i\,x^{\nu_i}]$
for $i\geq 0$. Since  the coefficients of $P$ have non-negative order in $x$ and
$\nu_i\geq 0$, then the truncation up to order $N$ 
 of $R_i$ and that of $P_i$ coincide for all $i\geq 0$. 
In particular, if $\nu_k\leq N$ then, $\Phi_{(R_i,\nu_i)}(c_i)=0$ for
$0\leq i\leq k$.  
Since $\supp(R)$ is finite, it generates a finitely generated
semigroup $G$ of $\mathbb{R}_{\geq 0}$. As before, we may prove by
induction that $\nu_i\in G'$ for $0\leq i\leq k$.  
Since the set 
$G'\cap \mathbb{R}_{\leq  N}$ finite, then the set 
$\supp s(x)\cap \mathbb{R}_{\leq   N}$ is finite
for any $N$ and therefore either
$\lim \nu_i=\infty$ or  $\supp s(x)$ is finite; in both cases
$s(x)\in\Omega$  and we are done.
\end{proof}

\subsection{On the completion of admissible polynomials to solutions}
Grigoriev and Singer in \cite{Grigoriev-Singer} provide, in the case
of differential equations, a criterion
for an admissible polynomial to be the truncation of an actual
solution of $P=0$, based on the following definition. 
As before,  $r(x)=\sum_{i=0}^{k-1}c_i\,x^{\nu_i}$ is an
  admissible generalized polynomial for the equation $P=0$, and 
 we denote $P_0=P$ and $P_{i+1}=P_i[c_i\,x^{\nu_i}]$, for $0\leq i\leq
 k-1$.
 Let $V=\Bot( E_{\nu_{k-1}}(P_{k}) )$ be the bottom vertex
 of $E_{\nu_{k-1}}(P_{k})$.

 \begin{definition}
  We say that \emph{$r(x)$ stabilizes $P$} if the height $\height(V)$
  is~$1$, and $\nu_{k-1}$ is greater than the maximum
  of the finite set $\Sigma=\{\mu\in \mathbb{R}_{\geq 0}\mid 
\Psi_{(P;V)}(\dmu{\mu})=0\}$.
\end{definition} 
In Section 3 of \cite{Grigoriev-Singer} the authors prove the
following \emph{stabilization criterion} (for differential equations):  if
$r(x)$ stabilizes $P$ then it is the truncation of a unique solution
$s(x)$ of $P=0$. For the benefit of the reader we include a proof
which covers also the $q$-difference case.

\begin{proof}[Proof of the stabilization criterion
  of~\cite{Grigoriev-Singer}]
  We need to show the existence and uni\-queness of a solution 
  $s(x)=r(x)+\sum_{i=k}^{\infty}c_i\,x^{\nu_i}\in\Omega$, with
  $\nu_i<\nu_{i+1}$ for all $i$.
  By Proposition \ref{pro:stabilization}, it is enough to
  prove that $r(x)$ can be completed uniquely to an $s(x)$ such
  that $s(x)$ satisfies the necessary  initial
  conditions~\eqref{eq:4}.

  Notice that whenever $\nu_k>\nu_{k-1}$ and $c_k$ satisfy the necessary initial
  condition $\Phi_{(P_k;\nu_k)}(c_k)=0$
  then the generalized polynomial $r(x)+c_k\,x^{\nu_k}$ stabilizes
  $P$. Thus, by recurrence, we only need to prove that:  either (I) $P(r(x))=0$
  and the condition
  $\Phi_{(P_k;\nu_k)}(c_k)=0$ does not hold for any  $\nu_k>\nu_{k-1}$
  and $c_k\neq 0$; or (II) $P(r(x))\neq 0$  and there only exist one $\nu_k>\nu_{k-1}$
  and one $c_k\neq 0$ with $\Phi_{(P_k;\nu_k)}(c_k)=0$.



  Let $\nu_{k}>\nu_{k-1}$. Because $\height V=1$, 
  there are only three cases for the element $E_{\nu_k}(P_k)$:
  (a) $E_{\nu_k}(P_k)$ is a vertex of $\mathcal{N}(P_k)$ with ordinate zero;
  (b) $E_{\nu_k}(P_k)$ is the vertex $V$;
  (c) $E_{\nu_k}(P_k)$ is the side joining  $V$ and the vertex of
  $\mathcal{N}(P_k)$ on the horizontal axis, say
  $(\beta,0)$. 

  In case (a), $\Phi_{(P_k;\nu_k)}(C)$ is a non-zero
  constant polynomial, so that it has no roots. In case (b),  
  $\Phi_{(P_k;\nu_k)}(C)=\Psi_{(P_k,V)}(\dmu{\nu_k})\,C$  whose only root is
  $C=0$ as $\Psi_{(P_k,V)}(\dmu{\nu_k})=\Psi_{(P_{k-1},V)}(\dmu{\nu_k})\neq 0$
  because $r(x)$ stabilizes $P$.  
  In case (c) we have
  \begin{displaymath}
    \Phi_{(P_k;\nu_k)}(C)=\Psi_{(P_k,V)}(\dmu{\nu_k})\,C+{P_k}_{(\beta,(0,\ldots,0))},
  \end{displaymath}
both of whose coefficients are non-zero by assumption so that it has a
single root. 

If $P(r(x))=0$ then we are necessarily in case (b)
because  $\mathcal{C}(P_k)$ has no points with ordinate zero, 
whence (I) holds.

Assume that $P(r(x))=d\,x^{\gamma}+\cdots$ with $d\neq 0$,
hence $(\gamma,0)$ is a vertex of $\mathcal{N}(P_k)$.
As $V=(\alpha,1)$ for some $\alpha$ and
$V\in L_{\nu_{k-1}}(P_k)$, then $\gamma>\alpha+\nu_{k-1}$.
If $\nu_k\neq \gamma-\alpha$ then we are either in case (a) or case
(b), so that  $\Phi_{(P_k;\nu_k)}(C)$ has no non-zero roots.
As $\gamma-\alpha>\nu_{k-1}$,
there only remains the possibility $\nu_k=\gamma-\alpha$, giving case
(c) with $\beta=\gamma$, and 
$\Phi_{(P_k;\nu_k)}(C)$ has a single non-zero root, which is (II).
\end{proof}

The criterion presented in Theorem~\ref{the:main3} is of different
nature that this one: it only depends on the dimension of the
$\mathbb{Q}$-vector space spanned by the 
support of $r(x)$ and the order of the equation,
while the above criterion required knowing the whole equation $P$,
$r(x)$, and the substitution $P_k=P[r(x)]$. 

In Section~\ref{sec:example} we present an example of an
admissible generalized polynomial that
satisfies the criterion  of Theorem~\ref{the:main3} and that can be
completed in two different ways to a solution of the differential
equation so that the stabilization criterion of Grigoriev and Singer
does not apply.

\subsection{Autonomous ordinary differential equations of first order}
\label{subsec:autFirstOrderDiffODEs}
In this subsection we will apply our previous results to
the case of first order autonomous ordinary differential equations.
In  the next Theorem we prove that any solution $s(x)\in \Omega$ of that kind of
equation is a formal Puiseux series (statement (1)). We also provide alternative
proofs of two known results: Theorems 11 and 12 in~\cite{Cano2022137}
(statements (2) and (3)).


Let $P=0$ be a non-trivial an autonomous ordinary differential equation of
first order, that is, a first order differential equation invariant by
the translations $x\mapsto x+c$, $c\in \mathbb{C}$. Hence $P$ can be
written as a non-zero polynomial with constant
coefficients in $y_0$ and $y_1$, where $y_1$ refers to to the derivative
$\frac{d\,}{d\, x}$. 
Let us denote by $\hat{K}^{*}
=\cup_{d\geq  1}\mathbb{C}((x^{\frac{1}{d}\mathbb{Z}}))$ the field of formal
Puiseux power series.

\begin{theorem}\label{th:AutFirstOrderODEs}
  Let $P(y_0,y_1)=0$ be non-trivial an autonomous first order ordinary differential
  equation. Then
  \begin{enumerate}
  \item If $s(x)\in \Omega$ be a solution of the equation
    $P=0$, then $s(x)\in \hat{K}^{*}$. 
  \item If $s(x)\in \hat{K}^{*}$ is a solution of $P=0$, then $s(x)$
    is a convergent Puiseux power series.
  \item For any point $(x_0,c_0)\in \mathbb{C}^{2}$, there exists an
    analytic solution $s(x)$ of $P=0$ passing through the point $(x_0,c_0)$.
  \end{enumerate}
\end{theorem}
\begin{proof} Part (2) is consequence of our Theorem~\ref{the:main2}, see
  Corollary~\ref{co:corTh2FirstOrderConstant}.
  In order to prove parts 
  (1) and (3), let us note some specific properties of the Newton polygon
  of~$P$ in this case.

  Since $P$ is a polynomial with constant coefficients, we can write
  \begin{equation}\label{eq:AutFirstODE}
  P=\sum_{i=0}^{d_0}\sum_{j=0}^{d_1}
  P_{i,j}\,y_0^{i}\,y_{1}^{j},\quad P_{i,j}\in \mathbb{C}.
  \end{equation}
  Then any term
  $P_{i,j}\,y_0^{i}\,y_{1}^{j}$ corresponds to the point $(-j,i+j)$ in
  $\mathcal{C}(P)$, which provides a bijection between the non-zero terms of
  $P$ and $\mathcal{C}(P)$.
  Let $V\in \mathcal{C}(P)$
  be a point lying in some side or vertex  of $\mathcal{N}(P)$. There
  exist unique integers $0\leq i\leq d_1$ and $0\leq j\leq d_2$ such
  that $V=(-j,i+j)$. Then, the indicial polynomial $\Psi_{(P,V)}(T)=P_{i,j}\,T^{j}$, with
  $P_{i,j}\neq 0$ because $V\in \mathcal{C}(P)$.
  We are assuming that $P$ is effectively a differential equation,
  that is, $y_1$ appears in $P$, so that there is a point with
  strictly negative abscissa in $\mathcal{C}(P)$.
  This implies that for
  any $\mu\in \mathbb{R}$, the characteristic polynomial
  $\Phi_{(P;\mu)}(C)$ has a non-zero root if and only if $\mu=0$ or
  $\mu$ is the co-slope of a side of $\mathcal{N}(P)$.

  Let us prove part (1). Let $s(x)=c_0\,x^{\mu_0}+\cdots\in \Omega$,
  with $c_0\neq 0$, be a non-constant 
  solution of $P=0$. We may assume that $\mu_0\neq 0$, otherwise, we
  perform the change of variable $P[c_0]=P(y_0+c_0,y_1)$ to obtain a
  polynomial with constant coefficients  which has
  $s(x)-c_0=c_1\,x^{\mu_1}+\cdots $, with $\mu_1>0$, as solution. 
  By Lemma \ref{le:key_lemma_Newton_Polygon},
  $c_0$ is a non-zero root of the characteristic polynomial
  $\Phi_{(P;\mu_0)}(C)$. By the above properties of $\mathcal{N}(P)$,
  either $\mu_0=0$ or $\mu_0$ is the co-slope of a side of
  $\mathcal{N}(P)$. Since the sides of $\mathcal{N}(P)$ are rational,
  necessarily $\mu_0$ is a non-zero rational number. Hence
  $r(x)=c_0\,x^{\mu_0}$ is an admissible polynomial for $P=0$.
  Since $\supp P=\{0\}$ and $\langle \supp r(x)\rangle =\mathbb{Q}$,
  then $r(x)$ satisfies the hypothesis of
  Theorem~\ref{the:main3}. Hence, $r(x)$ is the truncation of a solution $s(x)$ of $P=0$. 
  By Theorem~\ref{the:main1},
  $\langle\supp s(x)\rangle=\mathbb{Q}$, hence $s(x)\in
  \mathbb{C}((x^{\mathbb{Q}}))$. Now,  the arguments
  given in Proposition~\ref{pro:stabilization} for the grid-based
  case, shows that $s(x)$ is grid-based, hence  $s(x)\in
  \hat{K}^{*}$.

  Let us prove part (3). Since $P$ is autonomous, we may assume that
  $x_0=0$. Performing the transformation $P(y_0+c_0,y_1)$,
  we may assume that $c_0=0$.  Write $P$ as
  in~\eqref{eq:AutFirstODE}. If $P_{0,0}=0$, then $s(x)=0$ is a
  solution and we are done. Assume that $P_{0,0}\neq 0$. Since $P$ is
  of order one, the Newton polygon $\mathcal{N}(P)$ has a vertex of the
  form $(-j,i+j)$, with $j\geq 1$. Hence $\mathcal{N}(P)$ has at least
  one side with co-slope $0<\mu_0\in \mathbb{Q}$. By the properties of
  $\mathcal{N}(P)$, the characteristic polynomial
  $\Phi_{(P;\mu_0)}(C)$ has a non-zero root $c_0$, and as in part (2),
  there exists a power series $s(x)=c_0\,x^{\mu_0}+\cdots \in \hat{K}^{*}$
  solution of $P=0$. By Corollary~\ref{co:corTh2FirstOrderConstant},
  $s(x)$ is convergent. 
\end{proof}

\section{Example}\label{sec:example}  
Consider the irrational number $\tau=\pi/2$ and the linear differential equations
\begin{align*}
  L_1 & = \tau\,y_0-y_1-\big((\tau-1)\,x+(\tau-2)\,x^{2}+(\tau-3)\,x^{3}\big),\\
  L_2 & = \tau\,y_0-y_1-\big((\tau-1)\,x+(\tau-2)\,x^{2}+(\tau-5)\,x^{5}\big),
\end{align*}
where $y_1$ refers to the Euler derivative $x\frac{d\,y(x)}{d\,x}$
and $y_0=y(x)$, for $y(x)\in \Omega$.
Set $P=L_1\,L_2+x^6\,y_0\,y_1$ and $r(x)=x+x^{\tau}$.
Here $L_1\,L_2$ is considered as a product of polynomials, not
as a composition of differential operators. More precisely
\begin{multline*}
  P=
  y_{1}^2+\tau^2\,y_{0}^2 +\left(-2\,\tau +x^6\right)\,y_{0}\,y_{1}\\
  +\left((2\,\tau-2)\,x+(2\,\tau-4)\,x^2
    +(\tau-3)\,x^3+(\tau-5)\,x^5\right)\,y_{1}\\
  +\left(
    2\,\tau\,(1-\tau)x+2\,\tau(2-\tau)\,x^{2}
    +\tau\,(3-\tau)\,x^3+\tau\,(5-\tau)\,x^5
  \right)\,y_{0}\\
  +(1-2\,\tau+\tau^{2})\,x^{2}+(4-6\tau+2\,\tau^{2})\,x^{3}+
  \cdots
  +(15-8\,\tau+\tau^2)\,x^8.
\end{multline*}
Let us verify
that $r(x)$ satisfies the criterion of Theorem~\ref{the:main3} for
$P=0$ but it can be continued to two
different solutions, so that $r(x)$ does not satisfy the stabilization
criterion. Recall that Theorem~\ref{the:main3} guarantees also the
convergence of those continuations.

As the order of the equation $P=0$ is one and $\tau$ is irrational we
only need to show that $r(x)$ is admissible. To this end,
define $P_0=P$, $P_1=P_0[x]$, $P_2=P_1[x^{\tau}]$.
Writing only the terms corresponding to the border of each Newton
polygon, we have 
  \begin{align*}
  P_{1}&= (\tau\,y_0-y_1)^{2}+
2\,\left(\tau-2\right)\,x^2\,\left(y_{1}-\tau\,y_{0}\right)+
\left(\tau-2\right)^2\,x^4+\cdots,\\
P_2&=(\tau\,y_0-y_1)^{2}+
2\,\left(\tau-2\right)\,x^2\,\left(y_{1}-\tau\,y_{0}\right)+
\left(\tau-2\right)^2\,x^4+\cdots.
  \end{align*}

From now on we refer the reader to Figure~\ref{fig:example-A1} for the
structure and elements of the Newton polygon of each the equations we
shall compute. 

The element $E_1(P_0)$ of co-slope $\nu_1=1$ 
is the side of $\mathcal{N}(P_0)$ with vertices $(0,2)$ and $(2,2)$.
Its characteristic polynomial is
$\Phi_{(P;1)}(C)=\tau^{2}(\tau^{-1}-1)(C-1)^{2}$,
whose only root is $c_1=1$. Hence the polynomial $r_1(x)=x$ is
admissible for $P$  but  does not
satisfy either the criterion of Theorem~\ref{the:main3} (it has no
irrational exponents), or the
stabilization criterion (because $\height \Bot(E_1(P_1))=2$).

Now for $P_1$ the indicial polynomial for the vertex 
$V=(0,2)$ of $\mathcal{N}(P_1)$ is
$\Psi_{(P_1,V)}(\mu)=\tau-\mu$, whose root is $\mu=\tau$. Set $\nu_2=\tau$.
Since $1<\tau<2$, the element $E_{\tau}(P_1)$ is $V$, and its characteristic
polynomial is  $\Phi_{(P_1;\tau)}(C)=0$ so that any $c_2\in
\mathbb{C}$ is a root. We choose $c_2=1$.
Thus  $r_2(x)=x+x^{\tau}$ is an
admissible polynomial for $P=0$. 
We have not reached the stabilization
step because $E_{\tau}(P_2)=\{V\}$ and the bottom vertex is $V$ which has
height~$2$ (see $\mathcal{N}(P_2)$ in Figure~\ref{fig:example-A1}).
However, we can already apply Theorem~\ref{the:main3} so that
$x+x^{\tau}$ can be continued to at least one solution of
$P=0$. Moreover, by Theorem~\ref{the:main2},
any solution which is a continuation of $x+x^{\tau}$ is convergent.

Let us try and continue the process to see when we reach the
stabilization step.

As $\mathcal{N}(P_2)$ has a unique side, of co-slope $2$,
$E_2(P_2)$, we need to compute the corresponding characteristic polynomial
\begin{displaymath}
  \Phi_{(P_2;2)}(C)=(\tau-2)^{2}+2\,(\tau-2)(2-\tau)\,C+(\tau-2)^{2}=(\tau-2)^{2}(C-1)^{2},
\end{displaymath}
whose only root is $c_3=1$ which makes $r_3(x)=x+x^{\tau}+x^{2}$ an
admissible polynomial. Let $P_3=P_2[x^{2}]$:
\begin{displaymath}
  P_3=(\tau\,y_0-y_1)^{2}+
\left(\tau-3\right)\,x^3\,\left(y_{1}-\tau\,y_{0}\right)+
     \left(\tau-4\right)^2\,x^8+\cdots,
\end{displaymath}
whose Newton polygon $\mathcal{N}(P_3)$ has $E_2(P_3)=\{V\}$. Again,
we still have not reached the stabilization step.

The Newton polygon $\mathcal{N}(P_3)$ has two sides: setting
$W=(3,1)$, they are $E_{3}(P_3)=[V,W]$ and
$E_5(P_3)=[W,(8,0)]$ of
co-slopes $3$ and $5$ respectively. 
Notice that Proposition~\ref{prop:main} applies to the
points $V$, $W$ and $(8,0)$  and as a consequence their
corresponding indicial polynomials are powers of $(\mu-\tau)$ up to a
constant factor:
\begin{displaymath}
  \Psi_{(P_3;V)}(\mu)=(\tau-\mu)^{2},\,\,\,
\Psi_{(P_3;W)}(\mu)=(\tau-3)\,(\mu-\tau),\,\,\,
\Psi_{(P_3;(8,0)}(\mu)=(\tau-4)^{2}\,(\mu-\tau)^{0}.
\end{displaymath}
This guarantees that the
characteristic polynomials $\Phi_{(P_3;3)}(C)$ and $\Phi_{(P_3;5)}(C)$
have non-zero roots. For co-slope $3$ we get 
\begin{displaymath}
  \Phi_{(P_3;3)}(C)=(\tau-3)^{2}\,C^{2}+(\tau-3)\,(3-\tau)\,C=(\tau-3)^{2}\,C\,(C-1), 
\end{displaymath}
and for  co-slope $5$,
\begin{displaymath}
  \Phi_{(P_3;5)}(C)=(\tau-3)\,(5-\tau)\,C+(\tau-4)^{2},
\end{displaymath}
whose single root is $d_5=\frac{(\tau-4)^{2}}{(\tau-3)(\tau-5)}$.
Since $\Phi_{(P_3;3)}(1)=0$ and $\Phi_{(P_3;5)}(d_5)=0$,
we can continue $r_3(x)$ with either $x^{3}$ or $d_5\,x^{5}$ to get
admissible polynomials.
Let $P_4=P_3[x^{3}]$ and $Q_4=P_3[d_5\,x^{5}]$:
\begin{align*}
  P_4&=(\tau\,y_0-y_1)^{2}
       -\left(\tau-3\right)\,x^3\,\left(y_{1}-\tau\,y_{0}\right)+
       x^{8}+\cdots,\\
  Q_4&=(\tau\,y_0-y_1)^{2}+
       \left(\tau-3\right)\,x^3\,\left(y_{1}-\tau\,y_{0}\right)+
       (1+\tau)\,x^{7+\tau}+\cdots.
\end{align*}
The structure of  $\mathcal{N}(P_4)$ guarantees 
that $z(x)=x+x^{\tau}+x^{2}+x^{3}$ stabilize $P$ because 
 $\Bot(E_3(P_4))=W$ has height $1$ and its indicial polynomial
$\Psi_{(P_4,W)}(\mu)= (\tau-3)(\mu-\tau)$ has no roots greater than~$3$. 
The stabilization criterion now implies that $z(x)$
can be completed to an actual solution $\bar{z}(x)$, necessarily convergent.
Similar arguments apply to the other choice $Q_4=P_3[d_5\,x^{5}]$ and
$w(x)=x+x^{\tau}+x^{2}+d_5\,x^{5}$, which can be continued to a
convergent solution $\bar{w}(x)$. For the sake of the reader we
include some of the first terms of $\bar{z}(x)$ and $\bar{w}(x)$:
\begin{align*}
  \bar{z}(x)&=x+x^{\tau}+x^{2}+x^{3}+
  \frac{1}{(\tau-5)(3-\tau)}\,x^{5}+\frac{\tau+1}{4(\tau-3)}\,x^{4+\tau}+\cdots,\\
  \bar{w}(x)&=x+x^{\tau}+x^{2}+d_5\,x^{5}-\frac{\tau+1}{4(\tau-3)}\,x^{4+\tau}
              +\frac{3}{(\tau-6)(\tau-3)}x^{6}
+\cdots
\end{align*}
 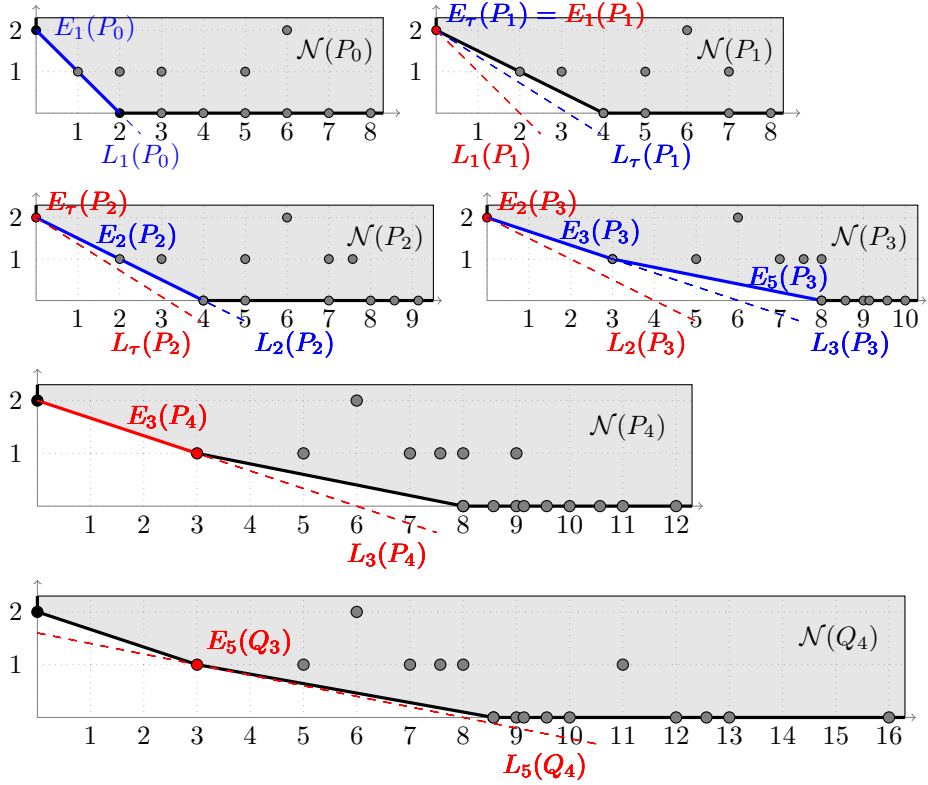
\begin{figure}[h!]
  \flushleft
  \begin{tikzpicture}[scale=0.55]
    \draw[fill=black!10, draw opacity=0.1] (0,2.3) -- (0,2) -- (1,1)
    -- (2,0) -- (8.3,0) -- (8.3,2.3) -- (0,2.3);
    \draw[gray, very thin,->](0,0) -- (0,2.6);
    \draw[gray, very thin,->](0,0) -- (8.7,0);
    \foreach \x in {1,...,8} {
      \draw (\x, 0) node[anchor=north] {$\x$};
    }
    \foreach \y in {1,...,2} {
      \draw (-0.1, \y) node[anchor=east] {$\y$};
    }
    \draw [gray, very thin, dotted](0,0) grid (8.3,2.3);
    \draw (8.3,1.5) node[anchor=east] {$\mathcal{N}(P_0)$};
    \draw[fill=black] (0,2) circle(3pt);
    \draw[fill=black] (2,0) circle(3pt);
    \draw [black, very thick] (0,2.3)--(0,2);
    \draw [blue, very thick] (0,2)--(2,0);
    \draw [black, very thick] (2,0)--(8.3,0);
    \draw [blue] (0.2,1.5) node[anchor=south west] {$E_{1}(P_0)$};
    \draw[fill=gray] (6,2) circle(3pt);
    \foreach \x in {1,2,3,5} {
      \draw [fill=gray] (\x,1) circle(3pt);}
    \foreach \x in {3,...,8} {
      \draw [fill=gray] (\x,0) circle(3pt);
    } 
     \draw [blue, dashed] (0,2)--(2.5,-0.5);
       \draw [blue] (2.5,-0.5) node[anchor=north]
       {$L_{1}(P_0)$};
     \end{tikzpicture}
   \begin{tikzpicture}[scale=0.55]
    \draw[fill=black!10, draw opacity=0.1] (0,2.3) -- (0,2) -- (2,1)
    -- (4,0) -- (8.3,0) -- (8.3,2.3) -- (0,2.3);
    \draw[gray, very thin,->](0,0) -- (0,2.6);
    \draw[gray, very thin,->](0,0) -- (8.7,0);
    \foreach \x in {1,...,8} {
      \draw (\x, 0) node[anchor=north] {$\x$};
    }
    \foreach \y in {1,...,2} {
      \draw (-0.1, \y) node[anchor=east] {$\y$};
    }
    \draw [gray, very thin, dotted](0,0) grid (8.3,2.3);
    \draw (8.3,1.5) node[anchor=east] {$\mathcal{N}(P_1)$};
    \draw [black, very thick] (0,2.3)--(0,2);
    \draw [black, very thick] (0,2)--(4,0);
    \draw [black, very thick] (4,0)--(8.3,0);
    \draw[fill=red] (0,2) circle(3pt);
    \draw[fill=black] (4,0) circle(3pt);
     \draw[fill=gray] (6,2) circle(3pt);
    \foreach \x in {2,3,5,7} {
      \draw [fill=gray] (\x,1) circle(3pt);}
    \foreach \x in {4,5,7,8} {
      \draw [fill=gray] (\x,0) circle(3pt);
       \draw [blue] (0,1.8) node[anchor=south west]
       {$E_{\tau}(P_1)=\color{red}E_{1}(P_1)$};
       \draw [blue, dashed] (0,2)--(3.93,-0.5);
       \draw [blue] (3.95,-0.5) node[anchor=north west]
       {$L_{\tau}(P_1)$};
       \draw [red, dashed] (0,2)--(2.5,-0.5);
       \draw [red] (2.5,-0.5) node[anchor=north east]
       {$L_{1}(P_1)$};
    }
  \end{tikzpicture}

   \begin{tikzpicture}[scale=0.55]
    \draw[fill=black!10, draw opacity=0.1] (0,2.3) -- (0,2) -- (2,1)
    -- (4,0) -- (9.5,0) -- (9.5,2.3) -- (0,2.3);
    \draw[gray, very thin,->](0,0) -- (0,2.6);
    \draw[gray, very thin,->](0,0) -- (9.7,0);
    \foreach \x in {1,...,9} {
      \draw (\x, 0) node[anchor=north] {$\x$};
    }
    \foreach \y in {1,...,2} {
      \draw (-0.1, \y) node[anchor=east] {$\y$};
    }
    \draw [gray, very thin, dotted](0,0) grid (9.3,2.3);
    \draw (9.5,1.5) node[anchor=east] {$\mathcal{N}(P_2)$};
    \draw [black, very thick] (0,2.3)--(0,2);
    \draw [blue, very thick] (0,2)--(4,0);
    \draw [black, very thick] (4,0)--(9.5,0);
    \draw[fill=red] (0,2) circle(3pt);
    \draw[fill=black] (4,0) circle(3pt);
     \draw[fill=gray] (6,2) circle(3pt);
    \foreach \x in {2,3,5,7,7.57} {
      \draw [fill=gray] (\x,1) circle(3pt);}
    \foreach \x in {4,5,7,8,8.57,9.14} {
      \draw [fill=gray] (\x,0) circle(3pt);
      \draw [blue] (1.2,1.0) node[anchor=south west] {$E_{2}(P_2)$};
       \draw [red] (0,1.8) node[anchor=south west]
       {$E_{\tau}(P_2)$};
       \draw [red, dashed] (0,2)--(3.93,-0.5);
       \draw [red] (3.95,-0.5) node[anchor=north east]
       {$L_{\tau}(P_2)$};
        \draw [blue, dashed] (0,2)--(5,-0.5);
       \draw [blue] (5,-0.5) node[anchor=north west]
       {$L_{2}(P_2)$};
    }
  \end{tikzpicture}
   \begin{tikzpicture}[scale=0.55]
    \draw[fill=black!10, draw opacity=0.1] (0,2.3) -- (0,2) -- (3,1)
    -- (8,0) -- (10.3,0) -- (10.3,2.3) -- (0,2.3);
    \draw[gray, very thin,->](0,0) -- (0,2.6);
    \draw[gray, very thin,->](0,0) -- (10.5,0);
    \foreach \x in {1,...,10} {
      \draw (\x, 0) node[anchor=north] {$\x$};
    }
    \foreach \y in {1,...,2} {
      \draw (-0.1, \y) node[anchor=east] {$\y$};
    }
    \draw [gray, very thin, dotted](0,0) grid (10.3,2.3);
    \draw (10.3,1.5) node[anchor=east] {$\mathcal{N}(P_3)$};
    \draw [black, very thick] (0,2.3)--(0,2);
    \draw [blue, very thick] (0,2)--(3,1);
    \draw [blue, very thick] (3,1)--(8,0);
    \draw [black, very thick] (8,0)--(10.3,0);
    \draw[fill=red] (0,2) circle(3pt);
    \draw[fill=black] (3,1) circle(3pt);
     \draw[fill=black] (8,0) circle(3pt);
     \draw[fill=gray] (6,2) circle(3pt);
    \foreach \x in {3,5,7,7.57,8} {
      \draw [fill=gray] (\x,1) circle(3pt);}
    \foreach \x in {8,8.57,9,9.14,9.57,10} {
      \draw [fill=gray] (\x,0) circle(3pt);
      \draw [blue] (1.5,1.1) node[anchor=south west] {$E_{3}(P_3)$};
      \draw [blue] (6,0.0) node[anchor=south west] {$E_{5}(P_3)$};
      \draw [red] (0,1.8) node[anchor=south west]
       {$E_{2}(P_3)$};
        \draw [red, dashed] (0,2)--(5,-0.5);
       \draw [red] (5,-0.5) node[anchor=north east]
       {$L_{2}(P_3)$};
           \draw [blue, dashed] (0,2)--(7.5,-0.5);
       \draw [blue] (7.5,-0.5) node[anchor=north west]
       {$L_{3}(P_3)$};
    }
  \end{tikzpicture} 

   \begin{tikzpicture}[scale=0.70]
    \draw[fill=black!10, draw opacity=0.1] (0,2.3) -- (0,2) -- (3,1)
    -- (8,0) -- (12.3,0) -- (12.3,2.3) -- (0,2.3);
    \draw[gray, very thin,->](0,0) -- (0,2.6);
    \draw[gray, very thin,->](0,0) -- (12.5,0);
    \foreach \x in {1,...,12} {
      \draw (\x, 0) node[anchor=north] {$\x$};
    }
    \foreach \y in {1,...,2} {
      \draw (-0.1, \y) node[anchor=east] {$\y$};
    }
    \draw [gray, very thin, dotted](0,0) grid (12.3,2.3);
    \draw (12,1.5) node[anchor=east] {$\mathcal{N}(P_4)$};
    \draw [black, very thick] (0,2.3)--(0,2);
    \draw [red, very thick] (0,2)--(3,1);
    \draw [black, very thick] (3,1)--(8,0);
    \draw [black, very thick] (8,0)--(12.3,0);
    \draw[fill=black] (0,2) circle(3pt);
    \draw[fill=red] (3,1) circle(3pt);
     \draw[fill=black] (8,0) circle(3pt);
     \draw[fill=gray] (6,2) circle(3pt);
    \foreach \x in {5,7,7.57,8,9} {
      \draw [fill=gray] (\x,1) circle(3pt);}
    \foreach \x in {8,8.57,9,9.14,9.57,10,10.57,11,12} {
      \draw [fill=gray] (\x,0) circle(3pt);
      \draw [red] (1.5,1.25) node[anchor=south west] {$E_{3}(P_4)$};
       \draw [red, dashed] (0,2)--(7.5,-0.5);
       \draw [red] (7.5,-0.5) node[anchor=north east]
       {$L_{3}(P_4)$};
     }
     \end{tikzpicture} 
   \begin{tikzpicture}[scale=0.70]
    \draw[fill=black!10, draw opacity=0.1] (0,2.3) -- (0,2) -- (3,1)
    -- (8.57,0) -- (16.3,0) -- (16.3,2.3) -- (0,2.3);
    \draw[gray, very thin,->](0,0) -- (0,2.6);
    \draw[gray, very thin,->](0,0) -- (16.5,0);
    \foreach \x in {1,...,16} {
      \draw (\x, 0) node[anchor=north] {$\x$};
    }
    \foreach \y in {1,...,2} {
      \draw (-0.1, \y) node[anchor=east] {$\y$};
    }
    \draw [gray, very thin, dotted](0,0) grid (16.3,2.3);
    \draw (16,1.5) node[anchor=east] {$\mathcal{N}(Q_4)$};
    \draw [black, very thick] (0,2.3)--(0,2);
    \draw [black, very thick] (0,2)--(3,1);
    \draw [black, very thick] (3,1)--(8.57,0);
    \draw [black, very thick] (8.57,0)--(16.3,0);
    \draw[fill=black] (0,2) circle(3pt);
    \draw[fill=red] (3,1) circle(3pt);
     \draw[fill=black] (8.57,0) circle(3pt);
     \draw[fill=gray] (6,2) circle(3pt);
    \foreach \x in {5,7,7.57,8,11} {
      \draw [fill=gray] (\x,1) circle(3pt);}
    \foreach \x in {8.57,9,9.14,9.57,10,12,12.57,13,16} {
      \draw [fill=gray] (\x,0) circle(3pt);
      \draw [red] (3,1) node[anchor=south west]{$E_{5}(Q_3)$};
        \draw [red, dashed] (0,1.6)--(10.5,-0.5);
       \draw [red] (10.5,-0.5) node[anchor=north east]
       {$L_{5}(Q_4)$};
          
    }
  \end{tikzpicture} 
 \caption{Newton polygons of $P_0,P_{1},P_2$ and $P_3$ following the
   admissible generalized polynomial $x+x^{\tau}+x^2+x^3$ up to it is
   reached the stabilization step.}
 \label{fig:example-A1}
\end{figure}

\section{Acknowledgments}
This work is part of action ``PID2022-139631NB-I00'', supported by the MICIU/AEI, with ID number 10.13039/501100011033, and by FEDER, UE.
\section{Competing Interests}
The authors declare that they have no competing interests.

\bibliographystyle{abbrv}

\end{document}